\documentclass[oneside,11pt]{amsart}
\usepackage{graphics}
\usepackage{graphicx}
\usepackage{tikz}
\usepackage{color}
\usepackage{yhmath}
\usepackage{xcolor}
\usepackage[utf8]{inputenc}
\usepackage[T1]{fontenc}
\usepackage{amsmath, mathrsfs, bbm}
\usepackage{enumitem}
\usepackage{tikz-cd}

\usepackage{url}

\usepackage{amssymb}
\usepackage{amsxtra}

\newtheorem{thm}{Theorem}[section]
\newtheorem*{IntroComb}{Theorem A}
\newtheorem*{IntroComb'}{Theorem A$^\prime$}
\newtheorem*{IntroDecomp}{Theorem B}
\newtheorem*{IntroInvLim}{Theorem C}
\newtheorem*{IntroComplete}{Theorem D}
\newtheorem*{IntroBowdComb}{Theorem E}
\newtheorem{prop}[thm]{Proposition}
\newtheorem{lem}[thm]{Lemma}
\newtheorem{cor}[thm]{Corollary}

\theoremstyle{definition}
\newtheorem{defn}[thm]{Definition}

\newtheorem{exmp}[thm]{Example}

\newtheorem{setup}[thm]{Setup}
\newtheorem{rem}[thm]{Remark}

\renewcommand{\bar}[1]{\overline{#1}}
\newcommand{\boundary}{\partial}

\renewcommand{\emptyset}{\varnothing}

\newcommand{\field}[1]{\mathbb{#1}}
\newcommand{\Z}{\field{Z}}

\newcommand{\N}{\field{N}}

\newcommand{\C}{\mathcal{C}}

\newcommand{\F}{\mathcal{F}}
\renewcommand{\P}{\mathcal{P}}

\newcommand{\inclusion}{\hookrightarrow}

\DeclareMathOperator{\Comp}{Comp}

\DeclareMathOperator{\Image}{Im}

\DeclareMathOperator{\Stab}{Stab}

\DeclareMathOperator{\diam}{diam}

\DeclareMathOperator{\Branch}{Branch}

\DeclareMathOperator{\valence}{Valence}

\DeclareMathOperator{\vertices}{Vert}
\DeclareMathOperator{\edges}{Edge}
\DeclareMathOperator{\graph}{\text{Graph}}
\DeclareMathOperator{\neck}{\text{Neck}}

\usepackage{combelow}

\newcommand{\Swiatkowski}{{\'{S}}wi{\k{a}}tkowski}

\usepackage{color}
\usepackage{pinlabel}
\usepackage{todonotes}
\usepackage{cleveref}

\usepackage{ifthen}

\newcommand{\showcomments}{no}

\newsavebox{\commentbox}
{\ifthenelse{\equal{\showcomments}{yes}}%
{\footnotemark
    \begin{lrbox}{\commentbox}
    \begin{minipage}[t]{1.25in}\raggedright\sffamily\upshape\tiny
    \footnotemark[\arabic{footnote}]}
{\begin{lrbox}{\commentbox}}}%
{\ifthenelse{\equal{\showcomments}{yes}}%
{\end{minipage}\end{lrbox}\marginpar{\usebox{\commentbox}}}
{\end{lrbox}}}

\begin{document}

\title[Parabolic cut pairs in Bowditch boundaries]{Parabolic cut pairs in boundaries of relatively hyperbolic groups}

\author[Kushlam Srivastava]{Kushlam Srivastava}
\address{Department of Mathematical Sciences\\
University of Wisconsin--Milwaukee\\
PO Box 413\\
Milwaukee, WI 53211\\
USA}
\email{kushlam@uwm.edu}

\begin{abstract}

Parabolic cut pairs in the boundaries of relatively hyperbolic group are a new and previously unexplored phenomenon.
In this paper, we give a way to create examples of relatively hyperbolic groups with parabolic cut pairs on their boundary via a combination theorem, which states that a group $G$, splitting as a graph of relatively hyperbolic groups with certain conditions, is relatively hyperbolic with inseparable parabolic cut pairs on the boundary $\boundary(G,\P)$. We also prove that all relatively hyperbolic groups with inseparable parabolic cut pairs in their boundaries arise via this combination theorem.

\Swiatkowski~gives a topological description of combining boundaries of vertex groups. Unfortunately, his method cannot be applied for fundamental reasons in this setting. We instead give two explicit topological descriptions of the boundary in terms of boundaries of the vertex groups.
 
\end{abstract}
\keywords{}

\subjclass[2020]{%
20F67 
20E08} 

\date{\today}

\maketitle

\setcounter{tocdepth}{2}

\section{Introduction}
\label{sec:Introduction}

Given a hyperbolic group $G$, we can learn a lot about it by exploring its Gromov boundary $\boundary G$. This boundary $\boundary G$ is a quasi-isometry invariant and topological properties of this boundary have been explored in detail. Importantly, due to Stalling's theorem $\boundary G$ is disconnected if and only if $G$ splits over a finite group. Swarup \cite{Swarup_cutpoints} proves that if $\boundary G$ is non-empty and connected then $\boundary G$ has no cut points.
Moreover, Bowditch \cite{Bowditch_localcutpoints} shows that for one ended $G$, cut pairs and local cut points on $\boundary G$ are directly related to JSJ splittings of $G$ over $2$--ended subgroups.

Similar to hyperbolic groups, Gromov \cite{essays} suggests the notion of a group $G$ to be hyperbolic relative to a collection of \emph{peripheral subgroups} $\P$ and Bowditch \cite{BowditchRelHyp} defines the boundary $\boundary(G,\P)$ of a relatively hyperbolic group generalizing the Gromov boundary $\boundary G$.
Points fixed by a peripheral subgroup under the action of $G$ on $\boundary(G,\P)$ are \emph{parabolic}. Bowditch \cite{Bowditch_local} and Dasgupta--Hruska \cite{DasguptaHruska_LC} prove that connected Bowditch boundaries are locally connected and cut points on them are parabolic.

Haulmark--Hruska \cite{HaulmarkHruska_JSJ} show that inseparable loxodromic cut pairs on the Bowditch boundary are directly related to JSJ splittings over $2$--ended subgroups. (A cut pair is \emph{inseparable} if it is not separated by any other cut pair and it is loxodromic if it is stabilized by a maximal virtually cyclic subgroup.) Thus it is natural to conjecture that like in the case of hyperbolic groups, inseparable cut pairs on the boundary of a relatively hyperbolic group are always related to splittings over $2$--ended subgroups. To this end, Haulmark \cite{HaulmarkRelHyp} tells us that inseparable cut pairs on the boundary are related to splittings over $2$--ended subgroups as long as the boundary does not have inseparable cut pairs consisting of parabolic points. Moreover, the boundary $\boundary(G,\mathcal{P})$ contains no inseparable parabolic cut pairs when the group $G$ is one ended \cite{HW}. However, as discovered by Hruska--Walsh \cite{HW}, it turns out that there are relatively hyperbolic groups with boundaries containing inseparable parabolic cut pairs.
The main theorems of this document are a combination theorem which generates boundaries with inseparable parabolic cut pairs and a decomposition theorem which states that inseparable parabolic cut pairs always arise via the before mentioned combination theorem. We note that these inseparable cut pairs are not captured by the JSJ splitting described by Haulmark--Hruska \cite{HaulmarkHruska_JSJ}.

\begin{IntroComb}[Creating Parabolic cut pairs]
\label{IntroComb}
Consider a finitely generated group $G$ acting without inversions and with finite quotient on a bipartite tree $T$ on two colors, red and white, such that the following holds.
\begin{enumerate}
    \item The red vertex stabilizers are finite.
    \item The white vertex stabilizers are relatively hyperbolic with connected boundaries without cut points.
    \item There is a ``signature'' (see Definition~\ref{signature}) for the action of $G$ on $T$. 
\end{enumerate}
Then the signature defines a parabolic forest (see Definition~\ref{forest}) $F$ on which $G$ acts on. Let $\P$ be the collection of stabilizers of connected components, called \emph{parabolic trees}, of $F$. Then $G$ is relatively hyperbolic to $\P$ and the Bowditch boundary $\boundary(G,\P)$ has inseparable parabolic cut pairs.
\end{IntroComb}

For a signature to exist, each edge stabilizer embeds into the intersection of two peripheral subgroups of the vertex stabilizer the edge is incident on.
The signature is a structure that keeps track of the two inclusions and that carries the instructions to glue boundaries of the vertex stabilizers inside $\boundary(G,\P)$. We also note that the connectivity conditions on the boundaries of the white vertex stabilizers can be relaxed to get a stronger theorem, Theorem~A$^\prime$,
which replaces the requirement that the white vertex stabilizers have connected boundaries without cut points with the weaker requirement that $\boundary(G,\P)$ is connected without cut points.

\begin{IntroComb'}
Consider a finitely generated group $G$ acting without inversions and with finite quotient on a bipartite tree $T$ on two colors, red and white, such that the following holds.
\begin{enumerate}
    \item The red vertex stabilizers are finite.
    \item The white vertex stabilizers are relatively hyperbolic.
    \item There is a ``signature'' (see Definition~\ref{signature}) for the action of $G$ on $T$. 
\end{enumerate}
Then the signature defines a parabolic forest $F$ (see Definition~\ref{forest}) on which $G$ acts on. Let $\P$ be the collection of stabilizers of connected components, called \emph{parabolic trees}, of $F$. Then $G$ is relatively hyperbolic to $\P$ and the Bowditch boundary $\boundary(G,\P)$ has inseparable parabolic cut pairs.
\end{IntroComb'}

We now give a decomposition theorem, which states that groups with inseparable parabolic cut pairs always arise via Theorem~A$^\prime$.

\begin{IntroDecomp}[Characterizing parabolic cut pairs]
\label{IntroDecomp}
Consider a finitely generated group $G$ hyperbolic relative to a collection $\P$ such that the boundary $\boundary(G,\P)$ is connected without cut points and has inseparable parabolic cut pairs. Then $G$ acts without inversions and with finite quotient on a bipartite tree $T$ on two colors, red and white, such that the following holds.
\begin{enumerate}
    \item The red vertex stabilizers are finite.
    \item The white vertex stabilizers are relatively hyperbolic.
    \item There is a ``signature'' (see Definition~\ref{signature}) for the action of $G$ on $T$. 
\end{enumerate}
Moreover, $\P$ arises from applying Theorem~A$^\prime$ to the above setup.
\end{IntroDecomp}

Kapovich--Kleiner asked the question of which spaces arise as boundaries of groups \cite{KapovichKleiner}. We explore this question by describing the boundary $\boundary(G,\P)$ in Theorem~A$^\prime$ in terms of boundaries of the vertex stabilizers
using two independent topological constructions described in Theorem~C and Theorem~D below.
The question is explored by \Swiatkowski\ \cite{Swiatkowski20} in the context of Gromov boundaries by constructing an inverse limit space called the limit of a tree system of spaces which is made by gluing vertex spaces in the pattern of a tree along spaces called edge spaces, and then compactifying.

One important condition in \Swiatkowski's construction is that it can only be applied in the case when the images of the adjacent edge spaces inside each vertex space are pairwise disjoint. While this condition may seem innocuous, it turns out to be critical in \Swiatkowski's construction (see Remark~\ref{\Swiatkowski rem}).
In this document, we modify \Swiatkowski's construction by allowing a vertex space to contain edge spaces that intersect with the restriction that edge spaces are made up of pairs of points. This small modification requires a surprisingly substantial amount of work to implement.
In Section~\ref{sec:kettlebell}, we construct the limit space of this tree system as an inverse limit.

\begin{IntroInvLim}
Given a tree system of cut pairs $\Theta$, we can construct an inverse limit $M_{\Theta}$ which is a compact space into which all vertex and edge spaces of $\Theta$ embed.
\end{IntroInvLim}

Suppose $\Theta$ is a tree system of cut pairs (see Definition~\ref{treesys}) with underlying tree $T$. In Section~\ref{sec:completion} we construct a space $M_T$ by gluing the vertex spaces of $\Theta$. We then construct a totally bounded metric for $M_T$ which we compactify by taking its completion.

\begin{IntroComplete}
Given a tree system of cut pairs $\Theta$ with underlying tree $T$, the completion $\bar{M}_T$ of $M_T$ is compact and all vertex and edge spaces of $\Theta$ embed isometrically into $\bar{M}_T$. Moreover, $\bar{M}_T$ is homeomorphic to the inverse limit $M_{\Theta}$ as described in Theorem~C.
\end{IntroComplete}

We now give a topological description of the boundary $\boundary(G,\P)$ obtained in Theorem~A$^\prime$ via the spaces described in Theorem~C and Theorem~D.

\begin{IntroBowdComb}
Assuming the setup of Theorem~A$^\prime$, define a tree system of cut pairs $\Theta$ over $T$ with the vertex spaces being Bowditch boundaries of the vertex groups. Then the Bowditch boundary $\boundary(G,\P)$ is homeomorphic to the inverse limit $M_{\Theta}$ as described in Theorem~C and the completion $\bar{M}_T$ as described in Theorem~D.
\end{IntroBowdComb}

\subsection{Outline of the paper}

In Section~\ref{sec:prelim} we give the basic definitions of a relative hyperbolic group and its properties along with the definition of a tree compatible metric spaces and the construction of its total space.
In Section~\ref{sec:graphcomb} we give the setup for the combination theorem, Theorem~A$^\prime$ including the definition of the signature. We then give a proof of part of Theorem~A$^\prime$ by constructing a fine hyperbolic graph.

In Section~\ref{sec:completion} we give the definition of the tree system of cut pairs and construct a total space associated to it. We then describe the completion of the total space which gives a proof of part of Theorem~D. We also describe topological properties of the completion including the structure of cut pairs.

Section~\ref{sec:kettlebell} gives a construction of an inverse system associated to a tree system of cut pairs. The factor spaces in the inverse system are kettlebell spaces constructed by attaching some line segments to the vertex spaces of the tree system. We describe the inverse limit of the inverse system which proves Theorem~C and show that it is homeomorphic to the completion described in Section~\ref{sec:completion}, which finishes the proof of Theorem~D.

Section~\ref{sec:homeo} completes the proof of Theorem~A$^\prime$ and Theorem~E. The proof involves describing the boundary $\boundary(G,\P)$ in Theorem~A$^\prime$ by constructing a tree system of cut pairs based on the setup described in Section~\ref{sec:graphcomb}.

Section~\ref{sec:decomp} gives a proof of the decomposition theorem, Theorem~B. We characterize all parabolic points on the boundary and construct the signature required to show that the relatively hyperbolic structure in the hypothesis of Theorem~B arises from Theorem~A$^\prime$.

\subsection{Acknowledgments} I would like to thank my advisor Chris Hruska for encouraging me to explore this phenomenon of parabolic cut pairs and for his invaluable feedback, guidance and comments throughout this project. I would also like to thank Genevieve Walsh and Jason Manning for helpful feedback on this document and Craig Guilbault for insightful discussions on properties on some of the spaces which arise as Bowditch boundaries.

\section{Preliminaries}
\label{sec:prelim}

This section collects background results on relatively hyperbolic groups, Peano continua and trees of spaces.

\subsection{Relatively hyperbolic groups}
A graph $K$ has vertex set $\vertices(K)$ and edge set $\edges(K)$.
A connected bipartite graph $K$ with $\vertices(K)=V\sqcup W$ is a \emph{star} if $|V|=1$. If $|W|=3$, it is a \emph{tripod}. The single element $v\in V$ is the \emph{center} of the star.

A \emph{path} consists of a sequence of vertices such that successive vertices are connected by an edge and a \emph{cycle} is a path $(x_0,x_1,\ldots, x_n)$ such that $x_0=x_n$. We regard two cycles as the same if their vertices are cyclically permuted.
A \emph{circuit} is a cycle $(x_0,x_1,\ldots, x_n)$ with distinct vertices except that $x_0=x_n$. The number of distinct vertices in a circuit is its \emph{length}.

\begin{defn}
\label{relhyp}
A graph $K$ is \emph{fine} if for all $e\in \edges(K)$ and $n>0$, the edge $e$ is contained in only finitely many circuits of length $n$.

Consider a finitely generated group $G$ with a $G$--set $\Lambda$. For each $x\in\Lambda$, let $P_x$ be its stabilizer, and let $\P=(P_x)_{x\in \Lambda}$ be an indexed family. Then the pair $(G,\P)$ is \emph{relatively hyperbolic} if $G$ acts with finite quotient on a connected, fine, hyperbolic graph $K$ such that $K$ has vertex set $\Lambda$, finite edge stabilizers, and finitely generated vertex stabilizers. Moreover, such a graph $K$ is a $(G,\P)$ graph. The subgroups in $\P$ are \emph{peripheral subgroups}.
\end{defn}

Two peripheral subgroups with distinct indices always have finite intersection.  In other definitions in the literature, $\P$ is defined as the collection of peripheral subgroups, while our definition with an indexed family allows more freedom in choosing $\P$, in particular allowing for multiple instances of the same (finite) peripheral group to appear in $\P$.
For example, consider a finite group $G$ acting trivially on a finite connected graph $K$. Then $(G,\P)$ is relatively hyperbolic where $\P=(P_x)_{x\in \vertices(K)}$ where each $P_x=G$.
Such an example would not be possible with existing definitions and is key in proving results in Section~\ref{sec:graphcomb}.

Traditionally hyperbolic groups are considered to be relatively hyperbolic with the empty indexed family $\P$. While such groups are studied in general contexts, they are not studied in this paper. As such we assume that for a relatively hyperbolic pair $(G,\P)$, the indexed family $\P$ is non-empty.

For a relatively hyperbolic pair $(G,\P)$, a subgroup $H$ of $G$ is \emph{parabolic} if it is a subgroup of a peripheral subgroup, and it is \emph{loxodromic} if it is a maximal virtually cyclic subgroup of $G$ and not parabolic.

\begin{defn}
\label{splitrelative}
Suppose $(G,\P)$ is relatively hyperbolic.  A \emph{finite splitting of $G$ relative to} $\P$ is a $G$--action on a simplicial tree $T$ without inversions such that each peripheral group fixes a vertex of $T$, and each edge stabilizer is finite.
\end{defn}

We now introduce one useful lemma due to Bowditch \cite{BowditchRelHyp}. The version stated here combines Lemmas 2.7 and 2.9 in \cite{MW}.

\begin{lem}
\label{gattach}
Let $K$ be a $(G,\P)$ graph for some relatively hyperbolic pair $(G,\P)$. If $e=(x,y)\notin \edges(K)$ add the edges $ge=(gx,gy)$ for all $g\in G$ to construct a new graph $K\cup Ge$. Then $K\cup Ge$ is also a $(G,\P)$ graph.
\end{lem}

\subsection{Bowditch boundaries and Peano continua}
\label{subsec: Bowditch boundaries}

Consider a fine hyperbolic graph $K$. We describe a topology on the set $\Delta K=\vertices(K)\sqcup \boundary K$. Consider $f\colon\N \rightarrow \N$ such that $f(n)\geq n$ for all $n\in \N$. An \emph{$f$--quasigeodesic arc} is an arc $\beta$ such that $\diam(\alpha)\leq f(d(x,y))$ for any subarc $\alpha$ of $\beta$ where $x,y$ are the endpoints of $\alpha$. 
For such an $f$, some finite set $A\subset \vertices(K)$ and $x\in \Delta K$, we define the sets $M_f(x,A)$ and $M'_f(x,A)$ as follows.

\begin{enumerate}
\item $M_f(x,A)$ consists of all $y\in \Delta K$ such that any $f$--quasigeodesic arc from $x$ to $y$ misses $A\setminus \{x\}$.
\item $M'_f(x,A)$ consists of all $y\in \Delta K$ such that there is an $f$--quasigeodesic arc from $x$ to $y$ missing $A\setminus \{x\}$.
\end{enumerate}

For any fixed $f$ as above, the collections $\{M_f(x,A)\}_{x,A}$ and $\{M'_f(x,A)\}_{x,A}$ form equivalent bases for a topology on the set $\Delta K$, where $x\in \Delta K$ and $A$ ranges over all finite subsets of $\vertices(K)$ \cite{BowditchRelHyp}. This topology on $\Delta K$ does not depend on the choice of $f$.
Note that a $1_{\N}$--quasigeodesic arc is a geodesic. Additionally, we write $M_{1_{\N}}(x,A)=: M(x,A)$ and $M'_{1_{\N}}(x,A)=: M'(x,A)$.

We now define the Bowditch boundary, a compact space associated with a given relatively hyperbolic pair. Consider the relatively hyperbolic pair $(G,\P)$ with a $(G,\P)$ graph $K$.
The Bowditch boundary $\boundary(G,\P)$ is then the space $\Delta K$ described above and it is independent of the choice of $K$ \cite{BowditchRelHyp}.

The boundary $\boundary(G,\P)$ is compact and metrizable \cite{BowditchRelHyp}. Moreover, $G$ acts on $\boundary(G,\P)$ via homeomorphisms. 
A point $p\in \Lambda\subset \boundary(G,\P)$ is \emph{parabolic} and $\Stab(p)$ acts properly discontinuously and cocompactly on $\boundary(G,\P)\setminus\{p\}$.

For a relatively hyperbolic pair $(G,\P)$ if $\P$ contains finite groups then the boundary $\boundary(G,\P)$ is disconnected and contains isolated points. If $\P$ does not contain finite groups then the boundary $\boundary(G,\P)$ is connected if and only if $G$ does not have a finite splitting relative to $\P$ \cite[Proposition~10.1]{BowditchRelHyp}, 
and if $\boundary(G,\P)$ is connected, it is locally connected \cite[Theorem~1.1]{DasguptaHruska_LC}. Thus $\boundary(G,\P)$ is a compact, connected, locally connected metrizable space which makes it a \emph{Peano continuum}.

Let $M$ be a Peano continuum without cut points. A closed set $S\subset M$ is said to be \emph{separating} if $M\setminus S$ is disconnected and $S$ \emph{separates} points in different components of $M\setminus S$. A point $x\in M$ is a \emph{cut point} of $M$ if $\{x\}$ is separating. A \emph{cut pair} $\{x,y\}$ is a separating set such that $x,y$ are not cut points. It is \emph{inseparable} if no other cut pair separates $x$ and $y$.
\begin{defn}
\label{cutpairtree}
Given a Peano continuum $M$ and 
a set of inseparable cut pairs $W$ on $M$, we define the \emph{inseparable cut pair tree $T_W$ dual to $W$}. 
Given $w,w_1$ and $w_2$ in ${W}$, we say $w$ is in \emph{between} $w_1$ and $w_2$ if $w$ separates at least one point of $w_1$ from at least one point of $w_2$.

A \emph{star} is a maximal subset $v\subset {W}$ such that for any $w_1$ and $w_2$ in $v$ there is no $w$ in ${W}$ in between $w_1$ and $w_2$. Let ${V}$ be the set of stars. We now define ${T}_W$ as a bipartite graph over the vertex set ${V}\sqcup {W}$ where the vertices $v\in {V}$ and $w\in {W}$ are joined by an edge if $w\in v$. The graph ${T}_W$ above is a simplicial tree due to Papasoglu--Swenson \cite[Theorem~6.6]{Papasoglu_Swenson}. A different proof of this fact was given by Hruska--Walsh \cite[Proposition~5.9]{HW} and a related result is also proved by Guralnik \cite[Theorem~3.15]{Guralnik}.
We also have that for $w,w_1,w_2\in W$, the cut pair $w$ is in between $w_1$ and $w_2$ if and only if $w$ is in on the shortest path from $w_1$ and $w_2$ in $T$.
\end{defn}

For a graph $K$, a subset $S\subset \vertices(K)$ is \emph{separating} if $K\setminus S$ is disconnected, where $K\setminus S$ is the graph $K$ with the vertices and edges involving elements in $S$ removed. For a space $X$, the collection of components of $X$ is \emph{$\Comp(X)$}.

\begin{lem}
\label{bowd disc}
Consider a relatively hyperbolic pair $(G,\P)$ with a $(G,\P)$ graph $K$ and connected boundary $\boundary(G,\P)$. Suppose $S\subset \vertices(K)$ is a finite separating set for $K$. Then $S\subset \boundary(G,\P)$ is a separating set for $\boundary(G,\P)$. 
\end{lem}
\begin{proof}
Let $C\in \Comp(K\setminus S)$. Let $C_{\Delta}\subset \Delta K=\boundary(G,\P)$ be the set containing $C$ and all points in $\boundary K$ which have representative geodesic rays in $K$ intersecting infinitely many points in $C$. Since for all $x\in C_{\Delta}$ we have $M(x,S)\subset C_{\Delta}$, we get that $C_{\Delta}$ is open. Moreover, as $\{C_{\Delta}\}_{C\in \Comp(K\setminus S)}$ forms a partition in $\boundary(G,\P)$, we get that $S$ is a separating set in $\boundary(G,\P)$.
\end{proof}

\subsection{Trees of metrizable spaces}
\begin{defn}
\label{treeofcompatiblemetricspaces}
A \emph{tree of metrizable spaces} $\mathcal{X}$ over a tree $T$ consists of the following data. 
\begin{enumerate}
    \item To each $v\in \vertices(T)$, we associate a non-empty metrizable space $X_v$ called a \emph{vertex space}.
    \item To each $e=(v,v')\in \edges(T)$, we associate a non-empty compact metrizable space $X_e$ called an \emph{edge space}, along with two embeddings $\psi_e^{v}\colon X_e\rightarrow X_v$ and $\psi_e^{v'}\colon X_e\rightarrow X_{v'}$.
\end{enumerate}

If vertex spaces are compact, $X$ is a \emph{tree of compact metrizable spaces}.

Consider an indexed family of metrics $\mathcal{D}=\{d_v\}_{v\in \vertices(T)}\bigcup \{d_e\}_{e\in \edges(T)}$ such that $d_v$ is a metric on $X_v$ for all $v\in \vertices(T)$ and $d_e$ is a metric on $X_e$ for all $e\in \edges(T)$. We say that $\mathcal{D}$ is \emph{compatible} on $\mathcal{X}$ if for all edges $e=(v,v')\in \edges(T)$ the embeddings $\psi_e^v$ and $\psi_e^{v'}$ are isometries with respect to the metric spaces $(X_v,d_v), (X_{v'},d_{v'})$ and $(X_e,d_e)$.
\end{defn}

Define the space $X_{\sqcup}:=\bigsqcup_{v\in \vertices(T)}X_v$. 
Consider $X_{\sqcup}$ along with a compatible family $\mathcal{D}$. We now define a metric $d_{\sqcup}\colon X_{\sqcup}\times X_{\sqcup}\rightarrow [0,\infty]$ as 
$$d_{\sqcup}(x,y)=\begin{cases}
    d_v(x,y) & \text{if there is some $v\in \vertices(T)$ with $x,y\in X_v$},\\
    \infty & \text{otherwise}.
\end{cases}$$

Consider $x,y\in X_{\sqcup}$. A \emph{linking chain} from $x$ to $y$ is a finite sequence $\alpha=(p_1,q_1,p_2,q_2,\ldots,p_m,q_m)$ in $X_{\sqcup}$ such that $p_1=x$, $q_m=y$ and $q_i\sim p_{i+1}$ for $i=1,2,\ldots,m-1$. For such an $\alpha$, define $d_{\sqcup}(\alpha)=\sum_{i=1}^md_{\sqcup}(p_i,q_i)$.

A linking chain $\alpha=(p_1,q_1,p_2,q_2,\ldots,p_m,q_m)$ between $x$ and $y$ is \emph{efficient} if $p_i,q_i\in X_{v_i}$ for $i=1,2,\ldots,k$, such that $(v_1,v_2,\ldots, v_{k-1}, v_k)$ is the shortest path between $v_1$ and $v_k$ in $T$.
If $\alpha$ is efficient then $d_{\sqcup}(\alpha)$ is always finite. Additionally, since the edge and vertex spaces are non-empty, there is always an efficient linking chain between $x$ and $y$.

Define the equivalence relation $\sim$ generated by $\psi_e^{v}(x)\sim \psi_e^{v'}(x)$ for all $e=(v,v')\in \edges(T)$ and $x\in X_e$. Define the \emph{quotient pseudometric} as
$$d_{\sim}(x,y)=\inf\{d_{\sqcup}(\alpha)\mid \alpha \text{ is a linking chain between $x$ and $y$}\}$$
for $x,y\in X_{\sqcup}$. 
Note that if $x\sim y$, then $d_{\sim}(x,y)=0$. To see this consider the linking chain $\alpha=(x,x,y,y)$ between $x$ and $y$. As $d_{\sqcup}(\alpha)=0$ we get $d_{\sim}(x,y)=0$.
The \emph{total space} of $\mathcal{X}$ is the space $X_T=X_{\sqcup}/{\sim}$. Let $d\colon X_T\times X_T\rightarrow[0,\infty)$ be induced by $d_{\sim}$.

\begin{prop}
\label{semimeteff}
For distinct $x,y\in X_{\sqcup}$ there is always an efficient linking chain $\alpha_0$ between $x$ and $y$ such that $d_{\sim}(x,y)=d_{\sqcup}(\alpha_0)$. Moreover, $d$ is a metric on $X_T$.
\end{prop}
\begin{proof}
Let $\alpha=(p_1,q_1,\ldots,p_k,q_k)$ be a linking chain from $x$ to $y$ such that $d_{\sqcup}(\alpha)<\infty$.
Then for $j=1,2,\dots, k$ there is some $v_j$ such that $p_j,q_j\in X_{v_j}$.

Consider $j$ such that $v_j$ and $v_{j+1}$ are separated by more than one edge in $T$. But since $q_j\sim p_{j+1}$ there is a sequence $q_j=z_0\sim z_1\sim z_2\sim \cdots\sim z_{\ell}=p_{j+1}$ such that $z_i$ and $z_{i+1}$ lie in adjacent vertex spaces. Then we can replace $\ldots,q_j,p_{j+1},\ldots$ by $\ldots, q_j=z_0, z_1,z_1, z_2,z_2.\ldots, z_{\ell-1}, z_{\ell-1}, z_{\ell}=p_{j+1},\ldots$ in $\alpha$ and $d_{\sqcup}(\alpha)$ would remain unchanged. So in the definition of $d_\sim$ it suffices to consider linking chains such that that $v_j$ and $v_{j+1}$ are separated by at most one edge in $T$ for all $j$.

Now suppose there is $i<j$ such that $v_i=v_{j}$. Then we note that we can replace $\ldots, p_i,q_i,\dots, p_{j},q_{j},\ldots$ by $\ldots, p_i,q_{j},\ldots$ in $\alpha'$ and by the triangle inequality $d_{\sqcup}(\alpha)$ would not increase. So we may further assume that all $v_j$'s are distinct and there is no backtracking \emph{i.e.}, $\alpha$ is efficient.

Assume we have a sequence $\alpha_n$ of efficient linking chains such that $d_{\sqcup}(\alpha_n)\rightarrow d_{\sim}(x,y)$. Let $\alpha_n=(p^n_1,q^n_1,\ldots,p^n_k,q^n_k)$. As the edge spaces are compact, by passing to a subsequence we can assume that $p^n_i\rightarrow p^0_i$ and $q^n_i\rightarrow q^0_i$ in $X_{v_i}$. Let $\alpha_0=(p^0_1,q^0_1,\ldots,p^0_k,q^0_k)$ to get $d_{\sim}(x,y)=d_{\sqcup}(\alpha_0)$, as needed. It follows that $d_{\sim}(x,y)=0$ if and only if $x\sim y$, proving that $d$ is a metric on $X_T$.
\end{proof}

Note that each vertex space and edge space embeds isometrically into $(X_T,d)$ via the quotient map. We also note that $d$ is maximal among all metrics on $X$ with the property that inclusions of all vertex spaces into $X_T$ are isometric embeddings.
Additionally, given a tree of spaces $\mathcal{X}$ over $T$ as above, for some subtree $S\subset T$ we can define a subtree of spaces $\mathcal{X}_S$ over $S$. Then, the total space $X_S$ isometrically embeds into $X_T$. 

For some subtree $S\subset T$, define $N_S$ as the set of edges $e\in \edges(T)\setminus \edges(S)$ adjacent to a vertex in $S$. Additionally, let $\mathcal{C}_S:=\{\psi_e^v(X_e)\mid e=(v,v')\in N_S, v\in \vertices(S) \}$ be the collection of images in $X_S$ of edge spaces in $N_S$. We then have the following result.

\begin{lem}
\label{neighborcpt}
Consider a tree of compact metrizable spaces $\mathcal{X}$ over the tree $T$ along with a compatible family of metrics $\mathcal{D}$. For a subtree $S\subset T$ such that $S\cup N_S=T$, if $X_S$ is compact, and $\{X_v\mid v\in \vertices(T\setminus S)\}$ forms a null collection in $X_T$, then $X_T$ is compact as well.
\end{lem}
\begin{proof}
For a sequence $(x_n)$ in $X_T$, we show that it has a convergent subsequence. Since $X_S$ is compact and each $X_v$ is compact, it suffices to assume that $x_n\in X_{v_n}$ for all $n$ such that the $v_n$'s are distinct and in $\vertices(T\setminus S)$. 
Choose some $a_n\in X_{v_n}\cap X_S$. Then since $(a_n)$ is a sequence in $X_S$, by passing to a subsequence we can assume that $a_n$ is convergent in $X_s$. Since $\diam(X_{v_n})\rightarrow 0$, we get that $x_n$ converges to the same limit as $a_n$.
\end{proof}

\section{Combining fine hyperbolic graphs}
\label{sec:graphcomb}
In this section we provide a combination theorem for relatively hyperbolic groups which generates parabolic cut sets (Theorem~\ref{combination}). 
In Section~\ref{sec:decomp} 
we focus on decomposing Bowditch boundaries along parabolic cut pairs and use the theorem in this section to glue back the boundaries. We note that Theorem~\ref{combination} generalizes a theorem of Bigdely--Wise \cite[Theorem~1.4]{BW} but the proof of Theorem~\ref{combination} is substantially more involved. 

\begin{defn}
\label{2admissible}
The action of a finitely generated group $G$ on a countable tree $T$ is said be \emph{admissible} if the following holds.
\begin{enumerate}
    \item \label{2admissible: finite quotient}
    $G$ acts on $T$ without inversions and with a finite quotient.
    \item \label{2admissible: T bipartite and valence 2}
    $T$ is bipartite over the vertex set $\vertices(T)={V}\sqcup {W}$ such that every vertex in $\vertices(T)$ has valence at least $2$.
    \item \label{2admissible: Lambda's}
    For each ${u}\in \vertices(T)$, the vertex stabilizer is $G_{{u}}$ and we have a $G_u$--set $\Lambda_u$.
    \item \label{2admissible: equivariance}
    For each $u\in \vertices(T)$ and $g\in G$, the sets $\Lambda_u$ and $\Lambda_{gu}$ are equivariantly equivalent with respect to $g$--conjugation. For $u\in \vertices(T)$, if $\P_u=\{\Stab_{G_u}(x)\}_{x\in \Lambda_u}$ then $(G_u,\P_u)$ is relatively hyperbolic.
    \item \label{2admissible: Gw finite}
    For each $w\in W$, we have $1\leq |\Lambda_w|<\infty$.
    If $|\Lambda_w|>1$, we have $|G_w|<\infty$.
    \label{admissible: degree}
\end{enumerate}
Moreover, the action is \emph{$n$--admissible} if $|\Lambda_w|=n$ for all $w\in W$ in ($\ref{admissible: degree}$). 
\end{defn}

We are mainly interested in $2$--admissible actions as they give rise to cut pairs. In particular, we aim to exploit the ambiguity explored in Technical Remark~1.2 in \cite{BW}. We define the signature of $G$ which encodes how we are gluing the peripheral structures of the vertex groups.
\begin{defn}
\label{signature}
A \emph{signature} $\mathcal{S}=\{s_{{e}}\}_{{e}\in \edges(T)}$ of an admissible action of $G$ on $T$ is an indexed family of injections satisfying the following conditions.
\begin{enumerate}
    \item For each $e=(v,w)\in \edges(T)$, the injection $s_e\colon  \Lambda_w\rightarrow \Lambda_v$ is
    such that for all $g\in G_v\cap G_w$ and $x\in \Lambda_w$, we have $g s_e(x)=s_e(gx)$. \label{signature: injection}
    \item For distinct edges $e,e'$, we have $\left| \Image(s_{e})\cap \Image(s_{e'})\right|\leq 1$.\label{signature: intersection}
    \item The indexed family $\mathcal{S}$ is $G$--equivariant.\label{signature: equivaraint}
\end{enumerate}
\end{defn}

Given a signature $\mathcal{S}$, fix any $v\in V$. A \emph{neck pair} in $\Lambda_v$ is a pair of points $(s_e(x),s_e(y))\in \Lambda_v\times\Lambda_v$ where $e=(v,w)$ incident on $v$ and $x,y\in \Lambda_w$. Let $\neck(v)$ be the collection of all such neck pairs in $\Lambda_v$.

\begin{lem}
\label{neckexistence}
For any fixed $v\in V$ there is a $(G_v,\P_v)$ graph $K_v$ such that for any neck pair $(x,y)\in \neck(v)$, there is an edge in between $x$ and $y$ in $K_v$.
\end{lem}
\begin{proof}
Since $v$ has finitely many orbits of edges incident on $v$ in $T$ and since the signature $\mathcal{S}$ is $G$--equivariant, $G_v$ acts on $\neck(v)$ with a finite quotient. Now, let $L_v$ be a $(G_v,\P_v)$ graph. Applying Lemma~\ref{gattach} finitely many times to $L_v$ we get the required graph $K_v$.
\end{proof}

For each fixed $v$, the graph $K_v$ in Lemma~\ref{neckexistence} is a \emph{reservoir} and for $e=(v,w)$ incident on $v$ and distinct $x,y\in \Lambda_w$, the \emph{neck edge} $n_e$ is the edge between $s_e(x)$ and $s_e(y)$ in $K_v$. 
For each $w\in W$, the \emph{junction edge} $j_w$ is the single edge in the $(G_w,\P_w)$ graph $K_w$.

\begin{defn}
\label{forest}
Consider a signature $\mathcal{S}$ of an admissible action of $G$ on $T$. The \emph{parabolic forest} $F$ for $\mathcal{S}$ is a graph consisting of the vertices $(u,x)$ for all $u\in \vertices(T)$ and $x\in \Lambda_u$. There are edges between $(w,x)$ and $(v,s_e(x))$ for all $e=(v,w)\in \edges(T)$ and $x\in \Lambda_w$.
Components of $F$ are \emph{parabolic trees}.
\end{defn}

We are now ready to state the main theorem for this section.

\begin{thm}
\label{combination}
Consider $G$ acting on $T$ with a $2$--admissible action. Let $\mathcal{S}=(s_e)_{e\in \edges(T)}$ be a signature for this action. Then the following hold.
\begin{enumerate}
\item 
\label{combination: Grelhyp}
Let $\Lambda=\sqcup_{u\in \vertices(T)} \Lambda_u$. Define an equivalence relation $\equiv$ on $\Lambda$ generated by $x\equiv s_e(x)$ for all $e=(v,w)\in \edges(T)$ and $x\in \Lambda_w$. Then $\bar{\Lambda}=\Lambda/{\equiv}$ is a $G$--set. Also define $\P=(\Stab(x))_{x\in \bar{\Lambda}}$. Then there is a graph $\bar{K}$ with $\vertices(\bar{K})=\bar{\Lambda}$ such that
the pair $(G,\P)$ is relatively hyperbolic and $\bar{K}$ is a $(G,\P)$ graph.
\item
\label{combination: Kbar}
For all $u\in V\sqcup W=\vertices(T)$, the subgraph $K_u\subset \bar{K}$ is convex.
\item 
\label{combination: paraforest}
Let $F$ be the parabolic forest for $\mathcal{S}$ and let $\Lambda_F$ be the collection of parabolic trees in $F$. Then there is a $G$--equivariant bijection between $\bar{\Lambda}$ and $\Lambda_F$ and $\P$ is the indexed collection of stabilizers of parabolic trees in $F$.
\item 
\label{combination: separating set}
If the Bowditch boundary $\boundary(G,\P)$ is connected, then for each $w\in W$, we have that $\Lambda_w$ is a separating set in $\bar{K}$ and in $\boundary(G,\P)$.
\end{enumerate}
\end{thm}

\begin{proof}
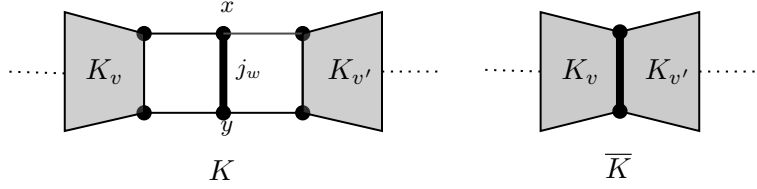
\begin{figure}
    \centering

\tikzset{every picture/.style={line width=0.75pt}} 

\begin{tikzpicture}[x=0.75pt,y=0.75pt,yscale=-1,xscale=1]

\draw [line width=3]    (160,40.83) -- (160,80.83) ;
\draw    (120,40.83) -- (120,80.83) ;
\draw [shift={(120,80.83)}, rotate = 90] [color={rgb, 255:red, 0; green, 0; blue, 0 }  ][fill={rgb, 255:red, 0; green, 0; blue, 0 }  ][line width=0.75]      (0, 0) circle [x radius= 3.35, y radius= 3.35]   ;
\draw [shift={(120,40.83)}, rotate = 90] [color={rgb, 255:red, 0; green, 0; blue, 0 }  ][fill={rgb, 255:red, 0; green, 0; blue, 0 }  ][line width=0.75]      (0, 0) circle [x radius= 3.35, y radius= 3.35]   ;
\draw    (200,40.83) -- (200,80.83) ;
\draw [shift={(200,80.83)}, rotate = 90] [color={rgb, 255:red, 0; green, 0; blue, 0 }  ][fill={rgb, 255:red, 0; green, 0; blue, 0 }  ][line width=0.75]      (0, 0) circle [x radius= 3.35, y radius= 3.35]   ;
\draw [shift={(200,40.83)}, rotate = 90] [color={rgb, 255:red, 0; green, 0; blue, 0 }  ][fill={rgb, 255:red, 0; green, 0; blue, 0 }  ][line width=0.75]      (0, 0) circle [x radius= 3.35, y radius= 3.35]   ;
\draw [color={rgb, 255:red, 0; green, 0; blue, 0 }  ,draw opacity=1 ]   (120,40.83) -- (160,40.83) ;
\draw [shift={(160,40.83)}, rotate = 0] [color={rgb, 255:red, 0; green, 0; blue, 0 }  ,draw opacity=1 ][fill={rgb, 255:red, 0; green, 0; blue, 0 }  ,fill opacity=1 ][line width=0.75]      (0, 0) circle [x radius= 3.35, y radius= 3.35]   ;
\draw [color={rgb, 255:red, 0; green, 0; blue, 0 }  ,draw opacity=1 ]   (120,80.83) -- (160,80.83) ;
\draw [shift={(160,80.83)}, rotate = 0] [color={rgb, 255:red, 0; green, 0; blue, 0 }  ,draw opacity=1 ][fill={rgb, 255:red, 0; green, 0; blue, 0 }  ,fill opacity=1 ][line width=0.75]      (0, 0) circle [x radius= 3.35, y radius= 3.35]   ;
\draw [color={rgb, 255:red, 74; green, 74; blue, 74 }  ,draw opacity=1 ]   (160,40.83) -- (200,40.83) ;
\draw [color={rgb, 255:red, 0; green, 0; blue, 0 }  ,draw opacity=1 ]   (160,80.83) -- (200,80.83) ;
\draw  [dash pattern={on 0.84pt off 2.51pt}]  (80,60.92) -- (50,60) ;
\draw  [dash pattern={on 0.84pt off 2.51pt}]  (240,60) -- (270,60) ;
\draw  [dash pattern={on 0.84pt off 2.51pt}]  (320,60) -- (290,59.08) ;
\draw  [dash pattern={on 0.84pt off 2.51pt}]  (400,60) -- (430,60) ;
\draw  [fill={rgb, 255:red, 128; green, 128; blue, 128 }  ,fill opacity=0.38 ] (120,40.83) -- (80,30) -- (80,90) -- (80,90) -- (120,80.83) -- cycle ;
\draw  [fill={rgb, 255:red, 128; green, 128; blue, 128 }  ,fill opacity=0.38 ] (240,30) -- (200,40.83) -- (200,80.83) -- (200,80.83) -- (240,90) -- cycle ;
\draw  [fill={rgb, 255:red, 155; green, 155; blue, 155 }  ,fill opacity=0.52 ] (400,30) -- (400,90) -- (360,80) -- (320,90) -- (320,30) -- (360,40) -- cycle ;
\draw [line width=3]    (360,40) -- (360,80) ;
\draw    (360,40) -- (360,80) ;
\draw [shift={(360,80)}, rotate = 90] [color={rgb, 255:red, 0; green, 0; blue, 0 }  ][fill={rgb, 255:red, 0; green, 0; blue, 0 }  ][line width=0.75]      (0, 0) circle [x radius= 3.35, y radius= 3.35]   ;
\draw [shift={(360,40)}, rotate = 90] [color={rgb, 255:red, 0; green, 0; blue, 0 }  ][fill={rgb, 255:red, 0; green, 0; blue, 0 }  ][line width=0.75]      (0, 0) circle [x radius= 3.35, y radius= 3.35]   ;

\draw (89,52.4) node [anchor=north west][inner sep=0.75pt]    {$K_{v}$};
\draw (211,52.4) node [anchor=north west][inner sep=0.75pt]    {$K_{v}{}_{'}$};
\draw (165,52.4) node [anchor=north west][inner sep=0.75pt]  [font=\footnotesize]  {$j_{w}$};
\draw (157,83.23) node [anchor=north west][inner sep=0.75pt]  [font=\footnotesize]  {$y$};
\draw (157,23.23) node [anchor=north west][inner sep=0.75pt]  [font=\footnotesize]  {$x$};
\draw (151,103.4) node [anchor=north west][inner sep=0.75pt]  [font=\normalsize]  {$K$};
\draw (329,52.4) node [anchor=north west][inner sep=0.75pt]    {$K_{v}$};
\draw (371,52.4) node [anchor=north west][inner sep=0.75pt]    {$K_{v}{}_{'}$};
\draw (351,99.4) node [anchor=north west][inner sep=0.75pt]  [font=\normalsize]  {$\overline{K}$};

\end{tikzpicture}

    \caption{The graph $K$ is constructed by adding pipe edges between a neighboring reservoir and junction. Then collapsing the pipe edges, we get $\bar{K}$.}
    \label{fig:reservoirsattach}
\end{figure}
We plan to construct $\bar{K}$ by first constructing a graph $K$ as follows.
\begin{enumerate}[label=(\roman*)]
    \item For each $v\in \vertices(T)$, add all vertices and edges of the reservoir graph $K_v$ as defined in Lemma \ref{neckexistence}.
    \item For each $w\in \vertices(T)$, add the $(G_w.\P_w)$ graph $K_w$ consisting of a single junction edge.
    \item For an edge $e=(v,w)\in \edges(T)$ and for each $x\in \Lambda_w$ add an edge $p_e(x)$ between the $x$ in $\Lambda_w$ and $s_e(x)$ in $\Lambda_v$. These $p_e$'s are \emph{pipe edges}. 
\end{enumerate}

A junction edge and a neck edge are \emph{neighboring} if their vertices are connected by a pair of complimentary pipe edges. Note that Definition~\ref{signature} ensures that a neck edge has a unique neighboring junction edge. Two neck edges are \emph{neighboring} 
if they share their unique neighboring junction edge.
 
Since $K_u$ is connected for all $u\in \vertices(T)$ and since $T$ is connected, we also have that $K$ is a connected graph along with an induced $G$--action.
Additionally, we can define a simplicial surjective map $\pi\colon K\rightarrow T$ such that for all $u\in \vertices(T)$, we have $\pi(K_u)=\{u\}$
and for all $e=(v,w)\in \edges(T)$ and $x\in \Lambda_w$, we have $\pi(p_e(x))=e$.

Observe that $\vertices(K)=\Lambda$ as described in the statement. Moreover, for $x,y\in \Lambda$ we have that $x\equiv y$ if and only if $x$ and $y$ are connected by a path only consisting of pipe edges in $K$. 

We now define the simplicial graph $\bar{K}$ with $\vertices(K)=\bar{\Lambda}$. We add an edge between $[x]_{\equiv}$ and $[y]_{\equiv}$ in $\bar{K}$ if and only if there is an edge between $x$ and $y$ in $K$ which is not a pipe edge or a neck edge (as shown in Figure~\ref{fig:reservoirsattach}). In other words, we obtain $\bar{K}$ from $K$ by collapsing pipe edges and removing neck edges. There are no duplicate edges because of Definition~\ref{signature}(\ref{signature: intersection}) which ensures that the resulting graph is simplicial. We note that $G$ acts on $K$ and $\bar{K}$ with finitely many $G$--orbits of edges. It follows that $\bar{K}$ is fine and hyperbolic from Lemma~\ref{kbarhyp} and Lemma~\ref{kbarfine} below. Defining $\bar{\Lambda}=\vertices(\bar{K})$ we get (\ref{combination: Grelhyp}) and from the construction of $\bar{K}$ we get (\ref{combination: Kbar}).

We now define a natural $G$--equivariant simplicial surjective map $\rho \colon K\rightarrow \bar{K}$ by defining $\rho(x)=[x]_{\equiv}$ for all $x\in \vertices(K)=\Lambda$. By abuse of notation, for $u\in \vertices(T)$, whenever we write $\Lambda_u\subset \bar{\Lambda}$ we mean the image of $\Lambda_u$ under the map $\rho$.
Images of junction edges under $\rho$ are junction edges in $\bar{K}$. Note that we also have a $G$--equivariant embedding $\iota \colon F\inclusion K$. Then we note that for the composition $\rho \circ \iota\colon F\rightarrow \bar{K}$ we get that that any parabolic tree maps to a single vertex in $\bar{K}$ and the pre-image of a vertex in $\bar{K}$ is a parabolic tree in $F$. Since $\rho \circ \iota$ is $G$--equivariant, we get (\ref{combination: paraforest}).
Since $\Lambda_w$ is separating set in $\bar{K}$ by construction, (\ref{combination: separating set}) follows from Lemma~\ref{bowd disc}.
\end{proof}

\begin{rem}
\label{combinationgeneralization}
Note that Theorem~3.5 holds even if the action of $G$ on $T$ is any admissible action with a similar proof. Moreover, if we restrict to just a $1$--admissible action then we get Theorem~1.4 from \cite{BW}.
\end{rem}

In order to complete the proof of Theorem~\ref{combination}, we first observe that $K$ is hyperbolic. The hyperbolicity of $K$ is a trivial case of \cite[Theorem~2.62]{kapovichsarkar}, which is a generalization of the Bestvina--Feighn combination theorem. However, the hyperbolicity of $K$ can be easily proved independent of the cited theorem which we leave as an exercise to the reader.

\begin{lem}
\label{khyp}
The graph $K$ in the proof of Theorem~\ref{combination} is hyperbolic. \qed
\end{lem}

Before proving that $\bar{K}$ is hyperbolic, we look at some properties of the map $\rho$.
The following result is due to Kapovich--Rafi \cite[Prop.~2.5]{KapovichRafi14}.

\begin{prop}
\label{kaporafi}
Consider a surjective Lipschitz map $\phi\colon X\rightarrow Y$ between connected graphs where $X$ is hyperbolic. Suppose there is $M>0$ such that
for any $x,x'\in \vertices(X)$ with $d_Y(\phi(x),\phi(x'))\leq 1$ and for any geodesic $\alpha$ in $X$ joining $x$ and $x'$, we have $\diam_Y(\phi(\alpha))\leq M$. Then $Y$ is hyperbolic.
\end{prop}

\begin{lem}
\label{parastronglyconv}
For distinct vertices $x_1,x_2$ in $K$ with $\rho (x_1)=\rho (x_2)$, there is a unique geodesic between $x_1$ and $x_2$ consisting only of pipe edges.
\end{lem}
\begin{proof}
Note that there is a path $\gamma$ without backtracking between $x_1$ and $x_2$ in $K$ consisting only of pipe edges. Let $n$ be the length of $\gamma$. Suppose $x_i\in K_{u_i}$ for $i=1,2$. Then $\pi(\gamma)$ is the shortest path between $u_1$ and $u_2$ in $T$ and it consists of $n$ edges. Let $\beta$ be any geodesic between $x_1$ and $x_2$. Then $\pi(\beta)$ is the shortest path between $u_1$ and $u_2$, which implies $\pi(\beta)=\pi(\gamma)$. Since only pipe edges survive under $\pi$, we must have that $\pi(\beta)$ contains at least $n$ pipe edges, which implies $\beta=\gamma$.
\end{proof}

\begin{lem}
\label{kbarhyp}
The graph $\bar{K}$ in the proof of Theorem~\ref{combination} is hyperbolic.
\end{lem}
\begin{proof}
By Proposition \ref{kaporafi} and Lemma \ref{khyp} above we only need to prove the following. For any $x,x'\in \vertices(K)$ and any geodesic $\alpha$ in $K$ joining $x$ and $x'$ such that either $\rho (x)=\rho (x')$ or $\rho (x)$ and $\rho (x')$ are joined by an edge in $\bar{K}$, we have $\diam(\rho (\alpha))\leq 1$. If $\rho (x)=\rho (x')$, from Lemma~\ref{parastronglyconv} we have that $\alpha$ consists only of pipe edges. Since $\rho$ collapses pipe edges, we have that $\rho (\alpha)=\rho (x)=\rho (x')$.

Otherwise, let $\bar{e}$ be the edge joining $\rho (x)$ and $\rho (x')$ in $\bar{K}$. Thus, by Lemma~\ref{parastronglyconv} we have that $\alpha$ is a (possibly empty) path of pipe edges followed by a lift of $\bar{e}$ and then followed by another (possibly empty) path of pipe edges.
Since $\rho$ collapses pipe edges, $\rho(\alpha)=\bar{e}$ and thus $\diam(\rho(\alpha))=1$.
\end{proof}

Now we prove that $\bar{K}$ is fine. Consider a circuit $c=(x_0,x_1,x_2,\ldots x_{n})$ in $\bar{K}$.
If for some $i,j$ with $1<j-i<n-1$ the vertices $x_i$ and $x_j$ are joined by an edge $e$ in $\bar{K}$, we can form smaller circuits $c_1=(x_0,x_1,\ldots,x_i,x_j\ldots,x_{n})$ and $c_2=(x_i,x_{i+1},\ldots, x_j,x_i)$ which are \emph{partitions} of $c$ along $e$.

Conversely, if there are two circuits $c=(x_1, x_2, \ldots, x_i,x_{i+1},\ldots x_n)$ and $c'=(x'_1,x'_2,\ldots, x'_j,x'_{j+1},\ldots, x'_m)$ such that $x_i=x'_{j+1}$ and $x_{i+1}=x'_{j}$ in $\bar{K}$ then we can form the \emph{concatenation} of $c_1$ and $c_2$ along $e=(x_i,x_{i+1})=(x'_{j+1},x'_j)$ as a bigger circuit
$$c=(x_0,x_1,\ldots,x_{i-1},x_i, x'_{j+2},\ldots, x'_m,x'_1,\ldots, x'_{j},x_{i},\ldots,x_n).$$
We note that concatenations and partitions are mutually inverse processes as seen in Figure~\ref{fig:partconcat}.

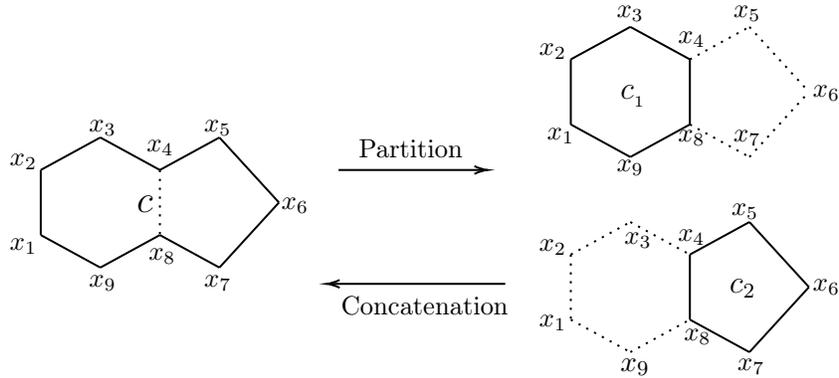
\begin{figure}
    \centering

\tikzset{every picture/.style={line width=0.75pt}} 

\begin{tikzpicture}[x=0.75pt,y=0.75pt,yscale=-1,xscale=1]

\draw  [dash pattern={on 0.84pt off 2.51pt}]  (91.09,91.44) -- (91.09,124.28) ;
\draw    (181.19,91.44) -- (254.28,91.44) ;
\draw [shift={(256.28,91.44)}, rotate = 180] [color={rgb, 255:red, 0; green, 0; blue, 0 }  ][line width=0.75]    (6.56,-1.97) .. controls (4.17,-0.84) and (1.99,-0.18) .. (0,0) .. controls (1.99,0.18) and (4.17,0.84) .. (6.56,1.97)   ;
\draw    (31.02,91.44) -- (31.02,124.28) ;
\draw    (31.02,91.44) -- (61.06,75.03) ;
\draw    (31.02,124.28) -- (61.06,140.7) ;
\draw    (61.06,140.7) -- (91.09,124.28) ;
\draw    (61.06,75.03) -- (91.09,91.44) ;
\draw    (91.09,91.44) -- (121.13,75.03) ;
\draw    (91.09,124.28) -- (121.13,140.7) ;
\draw    (121.13,140.7) -- (151.16,107.86) ;
\draw    (121.13,75.03) -- (151.16,107.86) ;
\draw    (358.4,35.62) -- (358.4,68.46) ;
\draw    (298.33,35.62) -- (298.33,68.46) ;
\draw    (298.33,35.62) -- (328.36,19.21) ;
\draw    (298.33,68.46) -- (328.36,84.88) ;
\draw    (328.36,84.88) -- (358.4,68.46) ;
\draw    (328.36,19.21) -- (358.4,35.62) ;
\draw  [dash pattern={on 0.84pt off 2.51pt}]  (358.4,35.62) -- (388.43,19.21) ;
\draw  [dash pattern={on 0.84pt off 2.51pt}]  (358.4,68.46) -- (388.43,84.88) ;
\draw  [dash pattern={on 0.84pt off 2.51pt}]  (388.43,84.88) -- (418.46,52.04) ;
\draw  [dash pattern={on 0.84pt off 2.51pt}]  (388.43,19.21) -- (418.46,52.04) ;
\draw    (358.4,134.13) -- (358.4,166.97) ;
\draw  [dash pattern={on 0.84pt off 2.51pt}]  (298.33,134.13) -- (298.33,166.97) ;
\draw  [dash pattern={on 0.84pt off 2.51pt}]  (298.33,134.13) -- (328.36,117.71) ;
\draw  [dash pattern={on 0.84pt off 2.51pt}]  (298.33,166.97) -- (328.36,183.38) ;
\draw  [dash pattern={on 0.84pt off 2.51pt}]  (328.36,183.38) -- (358.4,166.97) ;
\draw  [dash pattern={on 0.84pt off 2.51pt}]  (328.36,117.71) -- (358.4,134.13) ;
\draw    (358.4,134.13) -- (388.43,117.71) ;
\draw    (358.4,166.97) -- (388.43,183.38) ;
\draw    (388.43,183.38) -- (418.46,150.55) ;
\draw    (388.43,117.71) -- (418.46,150.55) ;
\draw    (265.29,148.91) -- (177.19,148.91) ;
\draw [shift={(175.19,148.91)}, rotate = 360] [color={rgb, 255:red, 0; green, 0; blue, 0 }  ][line width=0.75]    (6.56,-1.97) .. controls (4.17,-0.84) and (1.99,-0.18) .. (0,0) .. controls (1.99,0.18) and (4.17,0.84) .. (6.56,1.97)   ;

\draw (14,124.47) node [anchor=north west][inner sep=0.75pt]  [font=\small]  {$x_{1}$};
\draw (14,81.75) node [anchor=north west][inner sep=0.75pt]  [font=\small]  {$x_{2}$};
\draw (54,65) node [anchor=north west][inner sep=0.75pt]  [font=\small]  {$x_{3}$};
\draw (83,75) node [anchor=north west][inner sep=0.75pt]  [font=\small]  {$x_{4}$};
\draw (112,65) node [anchor=north west][inner sep=0.75pt]  [font=\small]  {$x_{5}$};
\draw (151,104) node [anchor=north west][inner sep=0.75pt]  [font=\small]  {$x_{6}$};
\draw (112,143) node [anchor=north west][inner sep=0.75pt]  [font=\small]  {$x_{7}$};
\draw (84,129.43) node [anchor=north west][inner sep=0.75pt]  [font=\small]  {$x_{8}$};
\draw (54,143) node [anchor=north west][inner sep=0.75pt]  [font=\small]  {$x_{9}$};
\draw (190,74.3) node [anchor=north west][inner sep=0.75pt]  [font=\small] [align=left] {Partition};
\draw (285,68.64) node [anchor=north west][inner sep=0.75pt]  [font=\small]  {$x_{1}$};
\draw (281,26.13) node [anchor=north west][inner sep=0.75pt]  [font=\small]  {$x_{2}$};
\draw (320,7) node [anchor=north west][inner sep=0.75pt]  [font=\small]  {$x_{3}$};
\draw (351,20) node [anchor=north west][inner sep=0.75pt]  [font=\small]  {$x_{4}$};
\draw (379,7) node [anchor=north west][inner sep=0.75pt]  [font=\small]  {$x_{5}$};
\draw (419,47) node [anchor=north west][inner sep=0.75pt]  [font=\small]  {$x_{6}$};
\draw (379,70) node [anchor=north west][inner sep=0.75pt]  [font=\small]  {$x_{7}$};
\draw (351,70) node [anchor=north west][inner sep=0.75pt]  [font=\small]  {$x_{8}$};
\draw (320,85) node [anchor=north west][inner sep=0.75pt]  [font=\small]  {$x_{9}$};
\draw (281,163.34) node [anchor=north west][inner sep=0.75pt]  [font=\small]  {$x_{1}$};
\draw (281,125) node [anchor=north west][inner sep=0.75pt]  [font=\small]  {$x_{2}$};
\draw (324,121.48) node [anchor=north west][inner sep=0.75pt]  [font=\small]  {$x_{3}$};
\draw (351,120) node [anchor=north west][inner sep=0.75pt]  [font=\small]  {$x_{4}$};
\draw (377.93,107) node [anchor=north west][inner sep=0.75pt]  [font=\small]  {$x_{5}$};
\draw (419,145) node [anchor=north west][inner sep=0.75pt]  [font=\small]  {$x_{6}$};
\draw (381,185) node [anchor=north west][inner sep=0.75pt]  [font=\small]  {$x_{7}$};
\draw (354,170) node [anchor=north west][inner sep=0.75pt]  [font=\small]  {$x_{8}$};
\draw (322,185) node [anchor=north west][inner sep=0.75pt]  [font=\normalsize]  {$x_{9}$};
\draw (78,103) node [anchor=north west][inner sep=0.75pt]  [font=\Large]  {$c$};
\draw (322,47.07) node [anchor=north west][inner sep=0.75pt]  [font=\large]  {$c_{_{1}}$};
\draw (377,145) node [anchor=north west][inner sep=0.75pt]  [font=\large]  {$c_{2}$};
\draw (181,153.76) node [anchor=north west][inner sep=0.75pt]  [font=\small] [align=left] {Concatenation};

\end{tikzpicture}

\caption{The solid lines in the figure represent part of a circuit and the dashed lines represent other edges in the graph. Going from left to right $c$ is partitioned into two circuits $c_1$ and $c_2$ while going from right to left $c_1$ and $c_2$ are concatenated to form $c$.}
\label{fig:partconcat}
\end{figure}

\begin{lem}
\label{nonatm}
A circuit $c$ is \emph{atomic} if there is some $v\in V$ such that $c\subset K_v$ in $\bar{K}$.
Given a non-atomic circuit $c$ in $\bar{K}$, there are circuits $c_1$ and $c_2$ in $\bar{K}$ such that $c$ is the concatenation of $c_1$ and $c_2$.
\end{lem}
\begin{proof}
Let $c'$ in $K$ be a circuit of minimal length such that $\rho(c')=c$.
(We consider the image of a path under a simplicial map to be a path by removing consecutive repeating vertices.)
Let
$$V_c:=\{v\in V\mid K_v\cap c'=\emptyset\}.$$
Note that $|V_c|\geq 2$ as otherwise $c$ would be atomic. Then there are $v_1,v_2\in V_c$ and $w\in W$ such that $v_i$ and $w$ are joined via an edge $e_i$ in $T$ for $i=1,2$. Recall that $j_w$ is the single edge in $K_w\subset K$. Then we can form partitions $c_1,c_2$ of $c$ along $\rho(j_w)$. Clearly, $c$ is the concatenation of $c_1$ and $c_2$.
\end{proof}

\begin{lem}
\label{kbarfine}
The graph $\bar{K}$ in the proof of Theorem~\ref{combination} is fine.
\end{lem}
\begin{proof}
Given an edge $e\in \edges(\bar{K})$ and $n>0$, let $C_{\bar{K}}(e,n)$ denote the collection of all circuits in $\bar{K}$ of length at most $n$ containing $e$.
We prove by induction that $|C_{\bar{K}}(e,n)|<\infty$ for any $n$ and for any edge $e$ in $\bar{K}$, which proves that $\bar{K}$ is fine.
Since $e$ is contained in finitely many reservoirs and since each reservoir subgraph is fine, there are finitely many atomic circuits in $C_{\bar{K}}(e,n)$.
By Lemma~\ref{nonatm}, any non-atomic circuit of length $n$ containing $e$ is made by concatenating smaller circuits. Such a circuit is made by concatenating a circuit in $C_{\bar{K}}(e,n-1)$ with another circuit of length at most $n-1$. By the induction hypothesis, $|C_{\bar{K}}(e,n-1)|<\infty$. Thus there are finitely many edges contained in circuits of $C_{\bar{K}}(e,n-1)$. Moreover for any edge $e'$ in a circuit of $C_{\bar{K}}(e,n-1)$, the set $C_{\bar{K}}(e',n-1)$ is finite. Thus, there are finitely many possible concatenations with circuits in $C_{\bar{K}}(e,n-1)$ to make a non-atomic circuit in $C_{\bar{K}}(e,n)$ which proves $|C_{\bar{K}}(e,n)|<\infty$.
\end{proof}

We now explore properties of the boundary $\boundary(G,\P)=\Delta\bar{K}$. Let $d$ be the metric on the graph $\bar{K}$.

\begin{lem}
\label{intersection of Delta Kv's}
Define the space $\Delta_{\sqcup}=\sqcup_{u\in \vertices(T)}\Delta K_u$. Then $\Lambda\subset \Delta_{\sqcup}$ where ${\Lambda}=\sqcup_{u\in \vertices(T)}\Lambda_u$. Extend the equivalence relation $\equiv$ as in Theorem~\ref{combination}(\ref{combination: Grelhyp}) from ${\Lambda}$ to $\Delta_{\sqcup}$.
Then the natural map $\sigma\colon \Delta_{\sqcup}/{\equiv} \to \Delta \bar{K}$ is injective. Moreover, $\Delta K_v$ is a subspace in $\Delta\bar{K}$ for each $v\in V$.
\end{lem}

\begin{proof}
Theorem~\ref{combination}(\ref{combination: Kbar}) gives us that $K_u$ is a convex subgraph in $\bar{K}$ for all $u\in \vertices(T)$. Therefore, $\Delta K_u$ is a subspace in $\Delta\bar{K}$ for all $u\in \vertices(T)$. We then have an obvious map $\sigma_{\cup}\colon\Delta_{\sqcup}\to  \Delta\bar{K}$. By Theorem~\ref{combination}(\ref{combination: Grelhyp}), the induced map $\sigma\colon\Delta_{\sqcup}/{\equiv}\to \Delta\bar{K}$ is well defined. 

To prove that $\sigma$ is injective we only need to prove that $\boundary K_v\cap \boundary K_{v'}=\emptyset$ for distinct $v,v'\in V$. The result then follows from Theorem~\ref{combination}(\ref{combination: Grelhyp}).
Consider some $x\in \boundary K_v$ and $x'\in \boundary K_{v'}$ in $\boundary\bar{K}$ along with sequences $(x_n)$ in $\Lambda_v$ and $(x'_n)\in \Lambda_{v'}$ converging to $x$ and $x'$ respectively in the visual topology. We show that $x\neq x'$.
Choose $w$ to be the closest vertex to $v'$ among the ones neighboring $v$.
Pick $z\in \Lambda_w$. Then since every path from $x_m$ to $x'_n$ has to go through $\Lambda_w$, it is straightforward to conclude that $(x_m,x_n)_z$ is bounded for all $m,n$,
which gives us that $x\neq x'$.
\end{proof}

For the rest of this section, we explore what other points are present in $\Delta\bar{K}$ apart from the points in $\Delta K_v$'s discussed above.

\begin{defn}
\label{redundantsig}
Let $\bar{\Lambda}$ be as in Theorem~\ref{combination}.
Consider an end $x\in \boundary T$. A sequence $(p_n)_{n\in \N}$ in $\bar{\Lambda}$ is an \emph{associated sequence for $x$} if for some representative ray $(v_1,w_1, v_2,w_2,\ldots)$ of $x$ where $v_i\in V$ and $w_i\in W$ we have $p_n\in {\Lambda}_{w_n}$ for all $n\in \N$.
We say an end $x$ is \emph{redundant} if there exists a point $p\in \bar{\Lambda}$, some $N>0$ and an associated sequence $(p_n)$ such that for all $n>N$, we have $p_n=p$.
Additionally, $p$ is the \emph{redundant point} associated with $x$.
\end{defn}

The set of redundant ends is $\boundary_R T$ and the set of \emph{non-redundant} ends is $\boundary_0 T$. We now show that the non-redundant ends are precisely the ends we need to add to make the boundary.

\begin{lem}
\label{redundant diverging}
Consider $x\in \boundary_0T$ with a representative ray $(v_1,w_1, v_2,w_2,\ldots)$ for $x$ and an associated sequence $(p_n)$ for $x$ such that $p_n\in \Lambda_{w_n}$. Then for some $z\in\Lambda_{v_n}$, we have that $d(z,p_n)\rightarrow \infty$.
\end{lem}
\begin{proof}
Since $d(z,\Lambda_{w_n})\leq d(z,p_n)$ for each $n$ we only need to prove that $d(z,\Lambda_{w_n})\rightarrow \infty$. Note that by construction of $\bar{K}$, for $n<m$ every path from $z$ to $\Lambda_{w_m}$ goes through $\Lambda_{w_n}$. Therefore, for $n<m$ we have that $d(z,\Lambda_{w_n})\leq d(z,\Lambda_{w_m})$. Thus we have that $d(z,\Lambda_{w_n})$ is a non-decreasing sequence in $\N$. It then follows that $d(z,\Lambda_{w_n})$ either goes to infinity or is eventually constant.
Assume for contradiction that $d(z,\Lambda_{w_n})$ is eventually constant. Then there is $N>0$ and $p\in \bar{\Lambda}$ such that $p\in \Lambda_{w_n}$ for all $n>N$. However this is a contradiction as this implies that $x$ is redundant.
\end{proof}

\begin{lem}
\label{assocseq in bdry}
Consider $x\in \boundary_0T$ and an associated sequence $(p_n)$ for $x$. Then $(p_n)$ converges to a point  $y\in \boundary\bar{K}$ under the visual topology.
Moreover, any other associated sequence for $x$ converges to $y$.
\end{lem}
\begin{proof}
There is a representative ray $(v_1,w_1, v_2,w_2,\ldots)$ in $T$ for $x$ such that $p_n\in \Lambda_{w_n}$. Choose $z\in \Lambda_{v_1}$. Since any path from $p_m$ to $z$ is within distance $1$ of $p_n$ when $m\geq n$ (and vice versa when $n\geq m$), using Lemma~\ref{redundant diverging} we have
$$(p_m,p_n)_z=\frac{1}{2}\left[d(p_m,z)+d(p_n,z)-d(p_m,p_n)\right]\to \infty\text{ as }m,n\to \infty.$$
Also observe that this limit is independent of the choice of $p_n$ in $\Lambda_{w_n}$ as $\diam(\Lambda_{w_n})=1$. This proves the second claim.
\end{proof}

\begin{lem}
\label{same comp not eq}
Consider $x\in \boundary_0T$ and an associated sequence $(p_n)$ for $x$ with a representative $(v_1,w_1, v_2,w_2,\ldots)$ in $T$ for $x$ such that $p_n\in \Lambda_{w_n}$ for all $n$. Also consider another sequence $(x_n)$ in $\bar{\Lambda}$.
Now suppose for some $z\in \Lambda_{v_1}$ and $N>0$, we have that $z$ and $x_n$ lie in the same component of $\bar{K}\setminus \Lambda_{w_N}$ for all $n$. Then $(x_n)$ and $(p_n)$ do not converge to the same point in $\boundary\bar{K}$.
\end{lem}

\begin{proof}
Consider $m>N$. Then $p_m$ and $x_n$ lie in different components of $\bar{K}\setminus \Lambda_{w_N}$. Then there is $q_N\in \Lambda_{w_N}$ such that $d(x_n,p_m)=d(x_n,q_N)+d(q_N,p_m)$. From the triangle inequality we get that 
$$d(z,x_n)\leq d(z,q_N)+d(q_N,x_n),\text{ and }d(z,p_m)\leq d(z,q_N)+d(q_N,p_m).$$
Thus we have that $(x_n,p_m)_z\leq d(z,q_N)$. Therefore $(x_n,p_m)_z$ is bounded for all $n$ and $m>N$, which implies $(x_n)$ is not equivalent to $(p_n)$.
\end{proof}

\begin{prop}
\label{non-red to bound k bar}
There is an injective map $\boundary \psi\colon \boundary_0T\rightarrow \boundary\bar{K}\setminus \Image(\sigma)$ defined as $\boundary\psi(x)=\lim_{n\rightarrow \infty}p_n$ where $(p_n)$ is an associated sequence for $x$ and $\sigma$ is as defined in Lemma~\ref{intersection of Delta Kv's}.
\end{prop}
\begin{proof}
Lemma~\ref{assocseq in bdry} gives us that $\boundary\psi$ is well defined. Now we prove $\boundary\psi(x)\notin\Image(\sigma)$ for $x\in \boundary_0T$. Consider $y\in \boundary K_v\subset \boundary\bar{K}$ for some $v\in V$ and a sequence $(y_n)$ in $\Lambda_v$ such that $y_n$ converges to $y$ in the visual topology. Then there is a representative ray $(v_1,w_1, v_2,w_2,\ldots)$ for $x$ such that $v_1=v$. Consider an associated sequence $(p_n)$ for $x$ such that $p_n\in \Lambda_{w_n}$. Lemma~\ref{same comp not eq} implies that $(p_n)$ and $(y_n)$ do not converge to the same point which proves $\boundary\psi(x)\neq y$.

Consider representative rays $(v_1,w_1, v_2,w_2,\ldots)$ and $(v'_1,w'_1, v'_2,w'_2,\ldots)$ for distinct $x$ and $x'$ in $\boundary_0T$ respectively such that $v_1=v'_1$ but $v_2\neq v'_2$. Consider associated sequences $(p_n)$ and $(p'_n)$ for $x$ and $x'$ respectively such that $p_n\in \Lambda_{w_n}$ and $p'_n\in \Lambda_{w'_n}$.
Since $w_2\neq w'_2$, by Lemma~\ref{same comp not eq} we get that $(p_n)$ and $(p'_n)$ do not converge to the same point which proves that $\boundary\psi$ is injective.
\end{proof}

\begin{lem}
\label{geodesics intersect lambda w}
Consider a geodesic ray $\alpha$ in $\bar{K}$ connecting $z\in \bar{\Lambda}$ to $\boundary\psi(x)$ for some $x\in \boundary_0T$. Also consider a representative ray $(v_1,w_1,v_2,w_2,\ldots)$ for $x$ such that $z\in \Lambda_{v_1}$. Then $\alpha$ intersects $\Lambda_{w_n}$ for all $n$.
\end{lem}
\begin{proof}
Consider a sequence $(p_n)$ such that $p_n\in \Lambda_{w_n}$ for all $n$. Then $(p_n)$ is an associated sequence for $x$. Assume for contradiction there is $N>0$ such that $\alpha$ does not intersect $\Lambda_{w_N}$. Then $z$ and $\alpha(n)$ are in the same component of $\bar{K}\setminus \Lambda_{w_N}$ for all $n$. Lemma~\ref{same comp not eq} then gives the required contradiction.
\end{proof}

\begin{lem}
\label{neighborhoods of non-redundant ends}
Consider $x\in \Image(\boundary\psi)$, and a neighborhood $M(x,A)$ in $\Delta\bar{K}$. Then there is $w\in W$ such that any geodesic ray from $a$ to $x$ intersects $\Lambda_w$ for any $a\in A$. It then also follows that $M(x,\Lambda_w)\subset M(x,A)$. 
\end{lem}
\begin{proof}
Let $y=(\boundary\psi)^{-1}(x)\in \boundary_0T$. For each $a\in A$, let $r_a^y=(v^a_1,w^a_1, v^a_2,w^a_2,\ldots)$ be a representative geodesic ray for $y$ such that $a\in \Lambda_{v^a_1}$. Since $A$ is finite, and $r_a^y$'s all point to $y$, we must have that $r=\bigcap_{a\in A}r_a^y$ is a geodesic ray representing $y$. 
Let $r=(v_1,w_1,v_2,\ldots)$ and $w=w_1$. Let $\alpha_a$ be the geodesic from $a$ to $x$. Then by Lemma~\ref{geodesics intersect lambda w} we get that $\alpha_a$ intersects $\Lambda_w$.
\end{proof}
\section{Joining metric compacta along cut pairs as a completion}
\label{sec:completion}
In this section, we aim to construct a compactification of a metric space obtained by gluing compact metrizable vertex spaces along cut pairs in the pattern of a bipartite tree $T$. First, we show that compatible shrinking metrics exist on the vertex spaces (Proposition~\ref{compatibleshrink}). Then we show that the space $M_T$ formed by gluing these vertex spaces is totally bounded (Proposition~\ref{completioncpt}). We then also describe its completion $\bar{M}_T$ as a compactum containing $M_T$ along with some ends of the tree $T$ (Proposition~\ref{completion}).
\begin{defn}
\label{treesys}
A \emph{tree system $\Theta$ of cut pairs} consists of the following data. 
\begin{enumerate}
    \item A countable bipartite tree $T$ with $\vertices(T)=V\sqcup W$ such that for all $w\in W$, we have $2\leq \valence(w)<\infty$. \
    \item \emph{Vertex Spaces}: To each $v\in V$, we associate a compact metrizable space $M_v$, called a \emph{constituent space}, and to each $w\in W$, we associate a discrete two point space $M_w$, called a cut pair space.
    \label{treesys: vertex}
    \item \emph{Edge Spaces}: To each edge $e=(v,w)\in \edges(T)$ where $v\in V, w\in W$, we associate an injection $i_e \colon M_w\rightarrow M_v$. For distinct $e,e'$, we have $|\Image(i_e)\cap \Image(i_{e'})|\leq 1$\label{treesys: edge}
    \item \emph{Edge Spaces form a null collection}: For each $v\in V$, define the \emph{peripheral collection} in $M_v$ as
    $$\C_v:=\{i_e(M_w) \mid e \text{ is an edge in $T$ incident on $v$}\}.$$
    For each $v\in V$, we have that $\mathcal{C}_v$ is a null collection in $M_v$.\label{treesys: nullity}
\end{enumerate}
\end{defn}

Note that given $\Theta$ as in Definition~\ref{treesys}, we can define a tree of compact metrizable spaces $\mathcal{M}$ over $T$ as in Definition~\ref{treeofcompatiblemetricspaces} as follows.
\begin{enumerate}
    \item For each $u\in \vertices(S)$ we let the vertex space be $M_u$.
    \item For each $e=(v,w)\in \edges(S)$, we let the edge space be $M_e=M_w$ and we let $\psi_e^w=id_{M_w}$ and $\psi_e^v=i_e$.
\end{enumerate}

We want to construct a metric space based
on a tree system $\Theta$ by picking carefully chosen metrics on the vertex spaces. 
In order to build this space we glue the constituent spaces along the cut pair spaces, in a natural way which extends the chosen metrics. We note that while the metrics on the vertex spaces above are arbitrary, we want to choose metrics so that they behave well with each other.

For some subtree $S\subset T$, consider the indexed family
$$\mathcal{D}_S=\{d_u\}_{u\in \vertices(S)}\bigcup \{d_e\}_{e\in \edges(S)}$$
such that $d_u$ is a metric on $M_u$ for all $u\in \vertices(S)$ and for an edge $e=(v,w)\in \edges(S)$, we have $d_e=d_w$. Note that we emit the edge metrics when writing out indexed families of metrics.

For $\Theta$ as above and for any subtree $S\subset T$ we can also define a \emph{subtree system of cut pairs} $\Theta_S$ with underlying tree $S$ and corresponding vertex and edge spaces from $\Theta$. 
Recall that if $\mathcal{D}_S$ is compatible as defined in Definition~\ref{treeofcompatiblemetricspaces}, we have the total metric space $(M_S,d_s)$ with the quotient metric $d_S$. 

\begin{defn}
\label{shrinking}
Given a subtree $S\subset T$, for a subtree system $\Theta_S$, the {compatible} collection $\mathcal{D_S}=\{d_u\}_{u\in \vertices(S)}$ is \emph{shrinking} with respect to some $v_0\in \vertices(S)\cap V$ if for any $v\in \vertices(S)\cap V$ at a distance $2k$ from $v_0$ in the tree $T$, we have that $\diam(M_v)\leq 2^{-k}$ in the total space $(M_S,d_S)$.
Moreover, we also say that $d_S$ is \emph{shrinking} on $M_S$.
\end{defn}
We now find metrics which are both compatible and shrinking.
\begin{prop}
\label{compatibleshrink}
Given a tree system $\Theta$ of cut pairs and some $v_0\in V$, there exists a compatible indexed family of metrics on the tree system $\Theta$ which is shrinking with respect to $v_0$.
\end{prop}

We need the following lemmas to prove the above, the first of which describes a distance non-increasing Urysohn function on a metric space. 
\begin{lem}
\label{distdec}
    Given two distinct points $p,q$ in a metric space $(X,d)$ consider the function $u\colon X \rightarrow [0, d(p,q)]$ defined as 
    $$u(x)=\displaystyle\frac{d(x,p)}{d(x,p)+d(x,q)}d(p,q).$$
    Then, $u$ is $1$--Lipschitz.
\end{lem}
\begin{proof}
Consider some $x,y\in X$. We now prove that $\left|u(x)-u(y)\right|\leq d(x,y)$. Without loss of generality, we assume $d(y,p)+d(y,q)\leq d(x,p)+d(x,q)$.
We then consider the following cases.

If $u(y)\leq u(x)$ then
\begin{align*}
 0\leq u(x)-u(y) &= \left(\displaystyle\frac{d(x,p)}{d(x,p)+d(x,q)}-\displaystyle\frac{d(y,p)}{d(y,p)+d(y,q)}\right)d(p,q) \\ 
 &\leq \left(\frac{d(x,p)-d(y,p)}{d(y,p)+d(y,q)}\right)d(p,q) \leq \frac{d(x,y)}{d(p,q)}d(p,q)=d(x,y)
\end{align*}
where the last inequality results from the triangle inequality.

If $u(y)\geq u(x)$ define another function $v\colon X\rightarrow [0,d(p,q)]$ as $v(x)=d(p,q)-u(x)$. 
Then applying the previous case to $v$ we get
$$0\leq u(y)-u(x)=v(x)-v(y)\leq d(x,y).$$
Thus, we get that $\left|u(x)-u(y)\right|\leq d(x,y)$, which is what we needed.
\end{proof}

We now see how we can keep the metrics of adjacent constituent spaces compatible while controlling the diameters of the spaces. This helps us join two adjacent constituent spaces along cut pairs, while making sure that the other cut pairs do not grow too far apart.

\begin{lem}
\label{halfer}
Given a compact metrizable space $X$, a null collection of pairs $\C$, a pair $C_0:=\{a,b\}\in \C$, and $K>0$ there exists a metric $d$ on $X$ such that $\diam(X)=K$ is attained by $d(a,b)$ and $\diam(C)\leq K/2$ for all $C\in \C\setminus\C_0$.
\end{lem}

\begin{proof}
Without loss of generality, assume $K=1$. Choose a metric $d'$ on $X$, such that $d'(a,b)=1$.
By nullity, there is some finite collection $\mathcal{A}\subset \C$ such that for all $C\in \C\setminus\mathcal{A}$ we have $\diam_{d'}(C)\leq 1/2$.

Define $u\colon X\rightarrow [0,1]$ as $u(x)={d'(x,a)}/({d'(x,a)+d'(x,b)})$. Then by Lemma~\ref{distdec},  
$\diam(u(C))\leq 1/2$ for all  $C\in \C\setminus\mathcal{A}$. 
Define $A:=\bigcup_{C\in \mathcal{A}}C\setminus C_0$,
$$\ell_1:=\min(u(A)\cup \{1/2\})\quad \text{and}\quad \ell_2:=\max(u(A)\cup \{1/2\}).$$

Let $f_1\colon[0,1]\rightarrow [0,1]$ be any continuous function such that
$$f_1([0,\ell_1])=[0,1/2],\quad f_1([\ell_1,\ell_2])=\{1/2\}, \quad \text{and} \quad f_1([\ell
_2,1])=[1/2,1].$$

Define $f=f_1\circ u\colon X\rightarrow[0,1]$. Then $f(a)=0, f(b)=1$ and $\diam(f(C))\leq 1/2$ for all $C\in \C\setminus C_0$. Now define the metric $d$ on $X\times [0,1]$ as
$$d((x_1,t_1),(x_2,t_2))=\max\Bigl\{\frac{d'(x_1,x_2)}{2\diam_{d'}(X)}, \left|t_1-t_2\right|\Bigl\}.$$
We now observe that $\graph(f)\subset X\times[0,1]$ is homeomorphic to $X$ and thus we can inherit the metric $d$ on $X$ from $X\times [0,1]$.
It is straightforward to check that this $d$ satisfies the required properties.
\end{proof}

Let $T_0=\{v_0\}\subset T$ and for $k>0$, let $T_k$ be the subtree in $T$ consisting of vertices at a distance less than or equal to $2k$ from $v_0$.
Now we give a proof for Proposition~\ref{compatibleshrink}.
\begin{proof}[Proof of Proposition~\ref{compatibleshrink}]
Since $M_{v_0}$ is compact, we can assume by scaling the metric on $M_{v_0}$ that $\diam(M_{v_0})\leq 1/2$.
Proceeding inductively, assume we have a collection of compatible metrics $\mathcal{D}_{T_k}$ on the tree system $\Theta_{T_k}$ which are shrinking with respect to $v_0$. Also assume that for any $w\in W$ such that $w\notin T_k$ and $w$ is connected to $T_k$ via a single edge $e'=(v',w)$ for some $v'\in \vertices(T_k)$, we have that $i_e(M_w)\subset M_{v'}$ has diameter less than $2^{-(k+1)}$. (Note that this is trivially true for $k=0$.) We now extend $\mathcal{D}_{T_k}$ to a collection of compatible metrics $\mathcal{D}_{T_{k+1}}$ on $\Theta_{T_{k+1}}$ which is shrinking with respect to $v_0$.

Now consider some $v\in V$ at a distance of $2(k+1)$ from $v_0$. Then there is some $v'\in \vertices(T_k)$ at a distance of $2$ from $v$ via the edges $e=(v,w)$ and $e'=(w,v')$, Let $d'=d_{v'}$ be the metric on $M_{v'}$ from $\mathcal{D}_{T_k}$. We now construct a metric $d=d_v$ on $M_v$. Let $M_{w}=\{x,y\}$ and define $d_w$ on $M_{w}$ as $d_w(x,y)=d'(i_{e'}(x),i_{e'}(y))$. Let $d$ be a metric on $M_v$ as given by Lemma~\ref{halfer} using $a=i_{e}(x)$, $b=i_{e}(y)\in M_v$ and $K=d_w(x,y)$. Note that $K\leq 2^{-(k+1)}$ by the induction hypothesis.

Thus the collection $\mathcal{D}_{T_{k+1}}$ constructed in this fashion is compatible and shrinking from Lemma~\ref{halfer}. Additionally for any $w'\in W$ such that $w\notin T_{k+1}$ and $w'$ is connected to $T_{k+1}$ via a single edge $e''=(v,w')$ for some $v\in \vertices(T_{k+1})$, we have that $i_{e'}(M_{w'})\subset M_{v}$ has diameter less than $2^{-(k+2)}$ from Lemma~\ref{halfer}. Thus the induction hypothesis is satisfied.
\end{proof}

We now give a setup which is used for the rest of this section.

\begin{setup}
\label{compsetup}
Let $\mathcal{M}$ be a tree of compatible metric spaces over $T$ with the vertex spaces and edge spaces as described in Definition~\ref{treesys} of a tree system $\Theta$ equipped with a collection of compatible shrinking metrics $\mathcal{D}=\{d_u\}_{u\in \vertices(T)}$, which exists by Proposition~\ref{compatibleshrink}.
Define $M_{\sqcup}=\sqcup_{u\in \vertices(T)}M_u$ and the equivalence relation $\sim$ on $M_{\sqcup}$ generated by $x\sim i_{e}(x)$ for all $e=(v,w)\in \edges(T)$ and $x\in M_w$. Define the total space as $M_T=M_{\sqcup}/\sim$.
We then have that $(M_T,d)$ is a metric space with the quotient metric $d$ as described in Proposition~\ref{semimeteff}. Note that $\{d|_{M_u}\}_{u\in \vertices(T)}=\mathcal{D}$.
Similarly, for some subtree $S\subset T$ we can define the subspace $M_S\subset M_T$.
\end{setup}

Now while the space $M_T$ contains all constituent spaces glued together as described by the tree system, it is not complete in general. Let $(\bar{M}_T,d)$ be the completion of the metric space $(M_T,d)$. Our goal is to prove that $\bar{M}_T$ is compact.
We also show that $\bar{M}_T$ is made by adding elements associated to some ends in $T$ to $M_T$.
First, we study the ends of $T$.
In particular, we first define sequences in $M_T$ which are associated to an end in $\boundary T $.

\begin{defn}
\label{assocseq}
Consider an end $x\in \boundary T$. A sequence $(p_n)_{n\in \N}$ is an \emph{associated sequence} for $x$ if for some representative ray $(v_1,w_1, v_2,w_2,\ldots)$ of $x$ where $v_i\in V$ and $w_i\in W$ we have $p_n\in M_{w_n}$ for all $n\in \N$.

\end{defn}
We now prove that every associated sequence is Cauchy and that any two sequences associated with the same end in $\boundary T$ eventually stay close.

\begin{lem}
\label{assocseqprops}
     Every associated sequence is Cauchy. Moreover, given two associated sequences $(p_n)$ and $(p'_n)$ for $x\in \boundary T$, we have $d(p_n,p'_n)\rightarrow 0$.
\end{lem}
\begin{proof}

Consider $x\in \boundary T$. Then we can represent $x$ by a ray $(v_0,w_0,v_1,v_2,\ldots)$ such that $v_i\in V$ and $w_i\in W$. Then consider an associated sequence $(p_n)_{n\in \N}$ such that $p_n\in M_{w_n}$. We now show that $(p_n)$ is a Cauchy sequence.

Given $\epsilon>0$, choose numbers $n<m$ such that $2^{-n+1}<\epsilon$. Define a subtree $F$ consisting of the path $(v_n, w_n, \ldots, w_{m-1},v_m)$. Then since $d$ is shrinking on $M_T$ and since $d$ is realized by an efficient linking chain by Proposition~\ref{semimeteff},
$$\diam(M_F)\leq 2^{-n}+2^{-n-1}+\cdots + 2^{-m}\leq 2^{-n+1}<\epsilon$$
in $M_T$. As $p_n,p_m\in M_F$, we get $d(p_n,p_m)<\epsilon$ which proves $(p_n)$ is Cauchy.

For the second assertion, passing to a subsequence and re-indexing we can assume there is a ray $(v_0,w_0,v_1,w_1,\ldots)$ with $p_n,p'_n\in M_{v_n}$.
Given $\epsilon>0$, choose $N$ such that $2^{-N}<\epsilon$. Then for all $n>N$ since $p_n,p'_n\in M_{v_n}$ and $\diam(M_{v_n})\leq 2^{-n}<\epsilon$, we have $d(p_n,p'_n)<\epsilon$ proving $d(p_n,p'_n)\rightarrow 0$.
\end{proof}

\begin{defn}
\label{redundant}
For a tree system $\Theta$ on cut pairs, we say an end $x$ is \emph{redundant} if there exists a point $p\in M_T$, some $N>0$ and an associated sequence $(p_n)$ such that for all $n>N$, we have $p_n=p$.
Additionally, $p$ is the \emph{redundant point} associated with $x$.
The set of redundant ends is $\boundary_R T$ and the set of non-redundant ends is $\boundary_0 T$.
\end{defn}
Note that by Lemma~\ref{assocseqprops} for some end $x\in \boundary T$ and any associated sequence $(p_n)_{n\in \N}$ we have that 
$\lim_{n\rightarrow \infty}p_n=p$ where $p$ is the associated redundant point. This also proves that the redundant point is unique.
Also note that this definition is very similar to Definition~\ref{redundantsig} which was for a redundant end with respect to a signature for a group splitting. 

Recall for a subtree $S\subset T$ that $N_S$ is the set of edges $e\in \edges(T)\setminus \edges(S)$ incident on a vertex in $S$.
\begin{defn}
\label{pericoll}
Given a tree system $\Theta$ of cut pairs over a tree $T$, for some subtree $S\subset T$ we define its \emph{peripheral collection} $\C_S=\{C_e\}_{e\in N_S}$ of two point subsets in $M_S$ as follows.
Consider some $e=(v,w)\in N_S$. 
If $w\in \vertices(S)$ but $v\notin \vertices(S)$
then define $C_e:=M_w$, and
if $v\in \vertices(S)$ but $w\notin \vertices(S)$ 
then define $C_e:=i_e(M_w)$.
\end{defn}

We now show that the non-redundant ends are in bijection with $\bar{M}_T\setminus M_T$.
First, we explore properties of the completion we explore some properties of the space $M_T$.

\begin{lem}
\label{Cknull}
The collection $\C_{T_k}$ is null in $M_{T_k}$.
\end{lem}
\begin{proof}
We prove the lemma using induction. For $k=0$, since $M_{T_0}=M_{v_0}$, the statement follows from Definition~\ref{treesys}(\ref{treesys: nullity}).
Assume the statement is true for $k-1$. We now prove it for $k$.
Recall that $M_{T_k}$ is obtained from $M_{T_{k-1}}$ by gluing $M_v$'s along corresponding pairs in $\C_{T_{k-1}}$, for all $v$ at a distance $2k$ from $v_0$ in $T$.
Consider some $\epsilon>0$. Since $\mathcal{C}_{T_{k-1}}$ is null in $M_{T_{k-1}}$, we have a finite subset $A\subset \mathcal{C}_{T_{k-1}}$, such that elements in $\mathcal{C}_{T_{k-1}}\setminus A$ have diameter less than $\epsilon$. By our choice of compatible metrics in the proof of Proposition~\ref{compatibleshrink}, the $M_v$'s glued to pairs in $\mathcal{C}_{T_k}\setminus A$ have diameter less than $\epsilon$. 

Let $M_{v_1},M_{v_2},\ldots,M_{v_n}$ be the spaces glued to the pairs from $A$ in $M_{T_k}$. Then $\mathcal{C}_A:=\cup_{i=1}^n \mathcal{C}_{v_k}$ contains all the pairs in $\C_{T_k}$ with diameter more than $\epsilon$. Since $\C_A$ is a finite union of null collections,  only finitely many of them have diameter more than $\epsilon$.    
\end{proof}
Applying Lemma~\ref{neighborcpt} with an induction argument, we get the following. 
\begin{lem}
\label{Tkcpt}
$M_{T_k}$ is compact for all $k\geq 0$. \qed 
\end{lem}

\begin{prop}
\label{completioncpt}
The metric space $(M_T,d_T)$ is totally bounded. Moreover, the completion $\bar{M}_T$ is compact.
\end{prop}
\begin{proof}
Note that we only need to prove that $(M_T,d_T)$ is totally bounded as the completion of a totally bounded space is compact. For $\epsilon>0$ there is $k>0$ such that $2^{-k+2}<\epsilon$. Consider the open cover $\{B_{\epsilon/2}(x)\mid x\in M_{T_k}\}$ of $M_{T_k}$. As $M_{T_k}$ is compact by Lemma~\ref{Tkcpt}, there is a finite subcover $\{B_{\epsilon/2}(x_i)\}_{i=1}^n$ for some $x_1,x_2,\ldots x_n\in M_{T_k}$.

We now prove that the collection $\{B_{\epsilon}(x_i)\}_{i=1}^n$ covers $M_T$.
Consider $x\in M_T\setminus M_{T_k}$. Let $a\in M_{T_k}$ be the closest point in $M_{T_k}$ to $x$. Then since $d$ is shrinking on $M_T$ and since $d$ is realized by an efficient linking chain by Proposition~\ref{semimeteff}, we have
$$d(a, x)\leq 2^{-k}+2^{-k-1}+\cdots =2^{-k+1}<\epsilon/2$$
in $M_T$. 
As $\{B_{\epsilon/2}(x_i)\}_{i=1}^n$ covers $M_T$ there is some $i$ such that $d(x_i, a)<\epsilon/2$. Thus, we get that $d(a,x)<\epsilon$.
\end{proof}

Now consider a subtree $S\subset T$. Since $M_S\subset M_T\subset \bar{M}_T$, as $\bar{M}_T$ is complete, $\bar{M}_S$ canonically embeds as a subspace in $\bar{M}_T$.
We now describe the completion $\bar{M}_T$ as a set and show that completing the space only adds the non-redundant ends. Before proceeding, we define the collection of branches glued to a space $M_S$ for some $M_T$.
\begin{defn}
\label{branch}
Consider a subtree $S\subset T$ and a connected component $S'$ of
$T\setminus S$ then $M_{S'}$ is a \emph{branch} glued to $M_S$ in $M_T$. The collection of branches glued to $M_S$ is $\Branch(S)$.
\end{defn}
We note that the set $\Branch(S)$ is in bijection with $N_S$.

\begin{lem}
\label{boundalpha}
There is a unique map $\boundary\alpha\colon \boundary T\rightarrow \bar{M}_T$ such that for every $x\in \boundary T$ and every associated sequence $(p_n)$ for $x$, we have that $p_n\rightarrow \boundary\alpha(x)$. Moreover, if $x$ is redundant then we have $\boundary\alpha(x)=p$ where $p$ is the redundant point associated with $x$ and if $x$ is non-redundant then $\boundary\alpha(x)\in \bar{M}_T\setminus M_T$.
\end{lem}

\begin{proof}
Let $p_n$ be an associated sequence for $x$. Define $\boundary \alpha(x)=\lim_{n\rightarrow \infty}p_n$ in $\bar{M}_T$. The first assertion then follows from Lemma~\ref{assocseqprops}. For the second assertion, if $x$ is redundant then the result follows from Lemma~\ref{assocseqprops} and Definition~\ref{redundant}. If $x$ is non-redundant, then for any $y\in M_T$ such that $y\in M_{v_y}$ for some $v_y\in V$ we can choose the associated sequence $(p_n)$ such that $p_1\in M_{v_y}$. Then $p_n\in M_{v_n}$ such that $v_n$ is at a distance of $2(n-1)$ from $v_y$ in $T$. As $d$ is the quotient metric, $d(p_n,y)$ is strictly increasing which implies that $p_{n}$ does not converge to $y$. Therefore, $\boundary\alpha(x)\in \bar{M}_T\setminus M_T$.
\end{proof}

\begin{prop}
\label{completion}
Consider the map $\alpha\colon M_T\sqcup \boundary T\rightarrow \bar{M}_T$ defined as the identity on $M_T$ and $\boundary\alpha$ on $\boundary T$, then we have that $\alpha$ is onto and $\alpha$ is a bijection when restricted to $M_T\sqcup \boundary_0 T$.
\end{prop}
\begin{proof}
We first show that $\alpha$ is onto. Consider $x'\in \bar{M}_T$ with a sequence $(x_n)_{n\in \N}$ in $M_T$ converging to $x'$ in $\bar{M}_T$. If infinitely many $x_n$ lie in some $M_{T_k}$, then since $M_{T_k}$ is compact by Lemma~\ref{Tkcpt}, we have that $(x_n)$ converges to some $x\in M_{T_k}\subset M_T$ and thus $\alpha(x)=x=x'$.

Now by passing to a subsequence, we assume that $x_k\in M_T\setminus M_{T_k}$ for all $k$. For every $k$, let $B_k$ be the set of branches in $\Branch(T_k)$ which intersect the sequence $(x_n)_{n=k}^{\infty}$. Consider the case when there exists $k$ such that the set $B_k$ is infinite. For some fixed $k$, by passing to a subsequence we can assume that for all $n\geq k$, no two $x_n$'s lie in the same branch of $\Branch(T_k)$. Let $M_n\in B_k$ be such that $x_n\in M_n$ and let $M_n$ be glued to the pair $\{a_n,b_n\}\in \C_{T_k}$.

Let $d_n=d(a_n,b_n)$. Since $\C_{T_k}$ is null in $M_{T_k}$ we have $d_n\rightarrow 0$. As $M_{T_k}$ is compact, by passing to a subsequence $a_n\rightarrow a$ for some $a\in M_{T_k}$. Then
$$d(a_n,x_n)\leq d_n+d_n/2+d_n/2^2+\cdots=2d_n\rightarrow 0.$$
Therefore, $x_n\rightarrow a$. Thus we have $\alpha(a)=a=x'$.

Now we assume that for all $k$ the set $B_k$ is finite. By passing to a subsequence, we can assume $B_k$ consists of one branch $M_k$ for all $k$.
Let $M_k$ be glued to the pair $\{a_k,b_k\}$ in $M_{v_k}\subset M_{T_k}$ along the edge $e_k=(v_k,w_k)$ in $T$.
Now consider $x=(v_0,w_0,v_1,w_1, \ldots)\in \boundary T$. Let $(p_k)_{k\in \N}$ be an associated sequence for $x$. Then since $d$ is shrinking on $M_T$ and since by Proposition~\ref{semimeteff}, $d$ is realized by an efficient linking chain we have
$$d(p_k,x_k)\leq 2^{-(k+1)}+2^{-(k+2)}+\cdots = 2^{-k}.$$
Thus $d(p_k,x_k)\rightarrow 0$ and since $x_k\rightarrow x'$ we have $p_k\rightarrow x'$ as well. Then we have $\alpha(x)=x'$.

We now prove that restricting $\alpha$ to $M_T\bigsqcup \boundary_0 T$, we have that $\alpha$ is bijective. We note that it is surjective as elements in $\boundary_R T$ map to $M_T$ by Lemma~\ref{boundalpha}. We prove that this restriction is injective. Consider distinct $x,x'\in \bar{M}_T$. The case when $x,x'\in M_T$ is straightforward. 
Now consider $x\in \boundary_0 T$ and $x'\in M_T$. Since $\alpha(x)\in \bar{M}_T\setminus M_T$ above, $\alpha(x)\neq \alpha(x')=x'$.

Now consider distinct $x,x'\in \boundary_0 T$. Consider the ray $(v_1,w_1,v_2,w_2,\ldots)$ representing $x$ and the ray  $(v'_1,w'_1,v'_2,w'_2,\ldots)$ representing $x'$ such that $v_1=v'_1$ but $v_n\neq v'_n$ for $n>1$. Then consider the associated sequences $(p_n)$ and $(p'_n)$ such that $p_n\in M_{w_n}$ and $p'_n\in M_{w'_n}$. Then since $v_n$ and $v'_n$ are $4(n-1)$ distance apart in $T$ and as $d$ is the maximal metric, $d(p_n,p'_n)$ is strictly increasing.
Thus, the sequences $(p_n)$ and $(p'_n)$ converge to different points in $\bar{M}_T$ which gives us that $\alpha(x)\neq \alpha(x')$.
\end{proof}

We now explore some connectivity properties of the space $\bar{M}_T$.

\begin{lem}
\label{T_k is connected without cut pairs}
Consider a tree system $\Theta$ on cut pairs such that the constituent spaces are connected without cut point and cut pairs. Also consider $k\geq 0$ and a subset $A\subset M_{T_k}$ such that $|A|\leq 2$ and $A\neq M_w$ for any $w\in W$. Then the subspace $M_{T_k}\setminus A$ is connected.
\end{lem}
\begin{proof}
We prove this using induction on $k$. By the hypothesis on the constituent space $M_{v_0}$, the statement holds for $k=0$. Now assume that $M_{T_{k-1}}\setminus A$ is connected. Let $V_k=\vertices(T_{k}\setminus T_{k-1})\cap V$. 
Then
$$M_{T_k}\setminus A=\left(M_{T_{k-1}}\setminus A\right)\bigcup\Biggl(\bigcup_{v\in V_k} M_v\setminus A\Biggr).$$
By the hypothesis on constituent spaces, $M_v\setminus A$ is connected for all $v\in V_k$. Moreover, for any $v\in V_k$ there is $w\in W$ such that $M_v\cap M_{T_{k-1}}=M_w$. Since $A\neq M_w$, we have that $M_{T_{k-1}}\setminus A$ intersects $M_v\setminus A$. The result then follows because $M_v\setminus A$ is connected for all $v\in V_k$ and $M_{T_{k-1}}\setminus A$ is connected by the induction hypothesis.
\end{proof}

\begin{prop}
\label{insepstr}
Consider a tree system $\Theta$ on cut pairs with underlying tree $T$, such that the underlying constituent spaces are connected without cut points and cut pairs. Let $M_T$ be the corresponding total space and let $\bar{M}_T$ be its completion. Then the following hold.
\begin{enumerate}
    \item \label{insepstr:no cut points} $\bar{M}_T$ is connected without cut points.
    \item \label{insepstr: insep cutpair}The cut pairs on $\bar{M}_T$ are precisely the subspaces $M_w$ for $w\in W$. Moreover, they are all inseparable.
    \item \label{insepstr: tree sys comp}
    For every $w\in W$, there is a bijection $\omega_w\colon \Comp(T\setminus w)\rightarrow \Comp(\bar{M_T}\setminus w)$ defined as $\omega_w(S)=\bar{M}_S\setminus w$.
    Moreover, for $w'\in W$ and $S\in \Comp(T\setminus w)$, we have that $w'\in S$ if and only if $M_{w'}\setminus M_{w}\subset \bar{M}_S$.
\end{enumerate}
\end{prop}
\begin{proof}
We first prove the following statement. For any subset $A\subset \bar{M_T}$ such that $|A|\leq 2$ and $A\neq M_w$ for any $w\in W$ we have that $\bar{M}_T\setminus A$ is connected. To see that this claim is true, first observe that $M_T=\bigcup_{k}M_{T_k}$. Then we note that ${M}_T\setminus A=\bigcup_{k}(M_{T_k}\setminus A)$ which is connected by Lemma~\ref{T_k is connected without cut pairs}. Since
$$M_T\setminus A\subset \bar{M}_T\setminus A\subset \text{cl}(M_T\setminus A)$$
we get that $\bar{M}_T\setminus A$ is connected, establishing (\ref{insepstr:no cut points}). We also get that if a two point set $A\neq M_w$ for any $w\in W$, then it is not a cut pair for $\bar{M}_T$.

Consider some $w\in W$ and let $M_w$ be the corresponding edge space. We now prove that $M_w$ is a cut pair. Let $S_1,S_2,\ldots ,S_n$ be the connected components of $T\setminus\{w\}$. We then note that $\bar{M}_T=\bigcup_{i=1}^n \bar{M}_{S_i}$ and that for $i\neq j$, we have $\bar{M}_{S_i}\cap \bar{M}_{S_j}=M_w$. Since $\bar{M}_{S_i}$ is compact by Proposition~\ref{completioncpt} for all $i$, we have that $\bar{M}_S$ is closed in $\bar{M}_T$. Thus, we have that $\bar{M}_{S_i}\setminus M_w$ is closed in $\bar{M}_T\setminus M_w$ for all $i$ proving that $M_w$ is a cut pair.

To see that $M_w$ is inseparable, consider another cut pair $M_{w'}\subset \bar{M}_T$ for some $w'\in W$. We assume that $M_{w}\cap M_{w'}=\emptyset$ otherwise $M_{w'}$ cannot separate $M_w$. Then there is a $v\in V$ such that $M_w$ is contained in $M_v$. Since the constituent spaces have no cut pairs,  $M_v\setminus M_{w'}$ is a connected subset containing $M_w$ proving that $M_w$ is not separated by $M_{w'}$. This proves (\ref{insepstr: insep cutpair}).

For some $w\in W$, consider $S\in \Comp(T\setminus w)$. Then applying (\ref{insepstr:no cut points}) and (\ref{insepstr: insep cutpair}) for the subtree system $\Theta_S$, we get that $\bar{M}_S\setminus w$ is connected in $\bar{M}_T\setminus w$. It is then straightforward to verify (\ref{insepstr: tree sys comp}).
\end{proof}

We now show that $T$ encodes the structure of the inseparable cut pairs.

\begin{prop}
\label{cutpairtree iso}
Consider a tree system $\Theta$ on cut pairs with underlying tree $T$ (bipartite on $V\sqcup W$), such that the constituent spaces are connected without cut points and cut pairs. Then for $\bar{M}_T$, let $W_I$ be the set of all inseparable cut pairs and let $T_I$ be the inseparable cut pair tree dual to $W_I$. Then there is a graph isomorphism $\phi_I\colon T\rightarrow T_I$ satisfying the following.
\begin{enumerate}
    \item $\phi_I(w)=M_w$ for all $w\in W$.
    \item $\phi_I(v)=\{M_w\mid w\text{ is adjacent to } v\}$ for all $v\in V$.
\end{enumerate}

\end{prop}
\begin{proof}
By Proposition~\ref{insepstr}(\ref{insepstr: insep cutpair}) the collection of inseparable cut pairs $W_I=\{M_w\}_{w\in W}$. Then from 
Definition~\ref{cutpairtree} we only need to prove the following. For $w_1,w_2\in W$, there is no inseparable cut pair in between $M_{w_1}$ and $M_{w_2}$ if and only if there is $v\in V$ with $w_1$ and $w_2$ adjacent to $v$ in $T$.

We first assume that there is no inseparable cut pair in between $M_{w_1}$ and $M_{w_2}.$ Assume for contradiction there is no $v\in V$ such that $w_1$ and $w_2$ are adjacent to $v$ in $T$. Then there is some $w\in W$ such that $w$ is in the path from $w_1$ to $w_2$. By Proposition~\ref{insepstr}(\ref{insepstr: tree sys comp}), there are connected components $S_1$ and $S_2$ in $T\setminus \{w\}$, such that $M_{w_1}\setminus M_w$ and $M_{w_2}\setminus M_w$ are contained in $\bar{M}_{S_1}$ and $\bar{M}_{S_2}$ respectively. Thus $M_{w_1}\setminus M_w$ and $M_{w_2}\setminus M_w$ are in different components of $\bar{M}_T\setminus M_w$ which is a contradiction.

Now we assume that there is $v$ adjacent to $w_1$ and $w_2$ in $T$.
Consider $w\in W$ such that $w\neq w_1,w_2$. Then ${w_1}$ and ${w_2}$ lie in $S$ for some connected component $S$ of $T\setminus \{w\}$.  Since $\bar{M}_S\setminus w$ is connected and contains $M_{w_1}\setminus M_w$ and $M_{w_2}\setminus M_w$, we have that $M_w$ cannot be in between $M_{w_1}$ and $M_{w_2}$.
\end{proof}
\section{Joining metric compacta along cut pairs as an inverse limit}
\label{sec:kettlebell}

In this section we consider a tree system $\Theta$ of cut pairs (as in Definition~\ref{treesys}) with underlying tree $T$ with a collection of shrinking compatible metrics $\mathcal{D}$ as in Setup~\ref{compsetup}. We then construct an inverse system $\mathcal{S}_{\Theta}$ (See Definition~\ref{invsys}.) and prove Theorem~\ref{completion is inverse limit} which states that the completion $\bar{M}_T$ of the total space $M_T$ is homeomorphic to the inverse limit $M_{\Theta}$ for $S_{\Theta}$. 

\begin{rem}
\label{\Swiatkowski rem}
\Swiatkowski\  in \cite{Swiatkowski20} defines a tree system of compact metric spaces similar to Definition~\ref{treesys} with the following differences. \Swiatkowski's definition allows for more generality by dropping the condition of certain vertex spaces to be discrete two point spaces. But, it is more restrictive as it requires edge spaces to be disjoint. In particular a tree system of cut pairs as in Definition~\ref{treesys} does not satisfy \Swiatkowski's
definition of a tree system.

If we try to implement \Swiatkowski's construction of an inverse system for a tree system of cut pairs $\Theta$ then the resulting inverse limit is in general not homeomorphic to the completion $\bar{M}_T$ which is our goal. In particular, \Swiatkowski's inverse limit may not contain an injective copy of each vertex space as illustrated by the following example.

Let $T$ be bipartite tree on $V\sqcup W$ such that each vertex has valence $3$. For each $v\in V$, the constituent space $M_v$ is a discrete space consisting of three points. In \cite{Swiatkowski20}, the factor spaces for the inverse system are constructed by collapsing the subsets in the peripheral collection $\C_F$ (See Definition~\ref{pericoll}.) in $M_F$.
In our example, after collapsing subsets in $\C_F$ for each $M_F$, each factor space consists of a single point and thus the inverse limit is a singleton. Clearly, the vertex spaces cannot inject into this singleton.
\end{rem}

For each $F\in \F$, we consider the space $M_F$ and the peripheral collection of edge spaces $\C_F$ as defined in Definition~\ref{pericoll}. We note that since $M_F$ is constructed by appropriately gluing together finitely many compact vertex spaces, it is straightforward to verify that $M_F$ is compact as well and that the collection $\C_F$ in $M_F$ is null.
We now define the factor spaces for our inverse system, the \emph{kettlebell space} $M_F^*$ for each $F\in \F$ as follows.

\begin{defn}
\label{kettlebell}
Define $T_F$ as a star graph 
with center $v_F$ and one vertex $v_C$ for each $C\in \C_F$. 
We now construct a tree of compatible metric spaces with underlying tree $T_F$. Let the vertex space $X_{v_F}$ be $M_F$ with the inherited metric $d$ from the space $(M_T,d)$ and for each $C=\{a,b\}\in \C_F$ let $X_{v_C}$ be the closed interval $[0,d(a,b)]$ which is the \emph{peripheral arc corresponding to $C$}. For an edge $e=(v_F,v_C)$ in $T_F$ where $C=\{a,b\}\in \C$, the edge space is $X_e=\{a,b\}$ with the metric $d_e(a,b)=d(a,b)$. Then $X_e$ embeds isometrically into $X_{v_F}$ by mapping $\{a,b\}$ identically inside $X_{v_F}$ and $X_e$ embeds isometrically into $X_{v_C}$ by mapping $a$ and $b$ to the endpoints of $[0,d(a,b)]$ in $X_{v_C}$.

We call the total space the \emph{kettlebell space} $M_F^*$. The quotient metric $d$ on $M_F^*$ extends the metric $d$ from $M_F$. As $M_F$ is made from gluing finitely many compact spaces, we get from Lemma \ref{neighborcpt} that $M_F^{*}$ is compact as well. 
\end{defn}

We now build an inverse system $\Theta^*$ on the directed set $(\F,\subset)$. With each $F\in \F$, we associate the kettlebell space $M_F^*$. For $F'\supset F$, we now construct bonding maps $f_{F'F} \colon M^*_{F'}\rightarrow M^*_F$. We first define a way to remove some interior peripheral arcs from kettlebell spaces.

For some subtree $F\in \F$ and some edge $e\in N_F$, let $M_F^*(e)$ be the space $M_F^*(e)$ with the interior of the peripheral arc corresponding to $e$ removed. Then $M_F^*(e)\subset M_F^*$ is a compact space.

\begin{defn}
\label{bondingaug}
If $F'=F\cup e$ for some $e=(v,w)\in N_F$ where $v\in V$ and $w\in W$ then $F'$ \emph{augments} $F$. We now define the bonding map $f_{F'F}$ in this case as illustrated in Figure~\ref{fig:kettlebell}.
If $v\in V(F)$ and $w\notin F$, it is straightforward to verify that $M_{F'}^*$ and $M_F^*$ are the same space. Then we define $f_{F'F}$ as the identity.
Now if $w\in V(F)$ and $v\notin V(F)$, we have a peripheral arc $\ell=[0, d(a,b)]$ glued to $a$ and $b$ in $M_F^*$
and $F'\setminus F$ consists of the vertex $v$. 

Now we note that $M_F^*(e)$ and $M_v^*(e)$ are closed subspaces in $M_{F'}^*$ such that $M_F^*(e)\cup M_v^*(e)=M_{F'}^*$ and $M_F^*(e)\cap M_v^*(e)=\{a,b\}$. Define $u_F(e)=u\colon M_v^*(e)\rightarrow [0, d(a,b)]=\ell\subset M_F^*$ as in Lemma~\ref{distdec}. We now define $f_{F'F}\big|_{M_F^*(e)}=id_{M_F^*(e)} $ and $f_{F'F}\big|_{M_v^*(e)}=u_F(e)$. Note that $f_{F'F}$ is well defined as $u\big|_{\{a,b\}}$ is the identity.
Thus, we get the full map $f_{F'F}$.
\end{defn}

\begin{figure}
    \centering

\tikzset{every picture/.style={line width=0.75pt}} 

\begin{tikzpicture}[x=0.75pt,y=0.75pt,yscale=-1,xscale=1]

\draw  [fill={rgb, 255:red, 155; green, 155; blue, 155 }  ,fill opacity=0.15 ] (89.47,61.05) .. controls (89.47,44.48) and (118.67,31.04) .. (154.68,31.04) .. controls (190.69,31.04) and (219.89,44.48) .. (219.89,61.05) .. controls (219.89,77.62) and (190.69,91.05) .. (154.68,91.05) .. controls (118.67,91.05) and (89.47,77.62) .. (89.47,61.05) -- cycle ;
\draw    (117.87,37.13) .. controls (119.32,-4.1) and (164.1,-0.29) .. (167,30.79) ;
\draw    (127.29,33.2) .. controls (130.19,14.3) and (159.9,1.61) .. (167,30.79) ;
\draw    (137.44,32.57) .. controls (147.58,7.32) and (164.97,19.37) .. (167,30.79) ;
\draw    (103.38,79.63) .. controls (42.15,88.04) and (57.18,66.84) .. (97.88,46.53) ;
\draw    (96.14,73.93) .. controls (58.09,74.08) and (70.22,64.94) .. (97.88,46.53) ;
\draw  [dash pattern={on 0.84pt off 2.51pt}]  (154.83,27.62) -- (163.74,27.24) ;
\draw    (146.13,30.79) .. controls (149.75,16.84) and (167.14,21.91) .. (167,30.79) ;
\draw    (210.91,46.52) .. controls (209.29,5.29) and (171.39,0.98) .. (168.16,32.06) ;
\draw    (205.11,41.45) .. controls (201.88,22.54) and (176.08,2.88) .. (168.16,32.06) ;
\draw    (199.31,38.91) .. controls (188,13.66) and (170.42,20.64) .. (168.16,32.06) ;
\draw    (189.89,35.74) .. controls (185.85,21.78) and (168,23.18) .. (168.16,32.06) ;
\draw  [dash pattern={on 0.84pt off 2.51pt}]  (171.78,30.03) -- (185.55,32.57) ;
\draw  [dash pattern={on 0.84pt off 2.51pt}]  (88.17,56.16) -- (85.27,68.22) ;
\draw  [fill={rgb, 255:red, 155; green, 155; blue, 155 }  ,fill opacity=1 ] (111.17,84.44) .. controls (78.26,94.89) and (48.13,95.03) .. (43.87,84.75) .. controls (39.61,74.47) and (62.83,57.66) .. (95.74,47.21) .. controls (69.35,57.62) and (52.17,71.74) .. (55.98,80.91) .. controls (59.78,90.08) and (83.29,91.27) .. (111.17,84.44) -- cycle ;
\draw    (235.11,59.08) -- (292.52,59.57) ;
\draw [shift={(294.52,59.59)}, rotate = 180.49] [color={rgb, 255:red, 0; green, 0; blue, 0 }  ][line width=0.75]    (10.93,-3.29) .. controls (6.95,-1.4) and (3.31,-0.3) .. (0,0) .. controls (3.31,0.3) and (6.95,1.4) .. (10.93,3.29)   ;
\draw    (120.05,86.61) .. controls (114.98,117.06) and (130.34,124.67) .. (141.06,91.05) ;
\draw    (125.84,88.51) .. controls (119.32,98.66) and (129.47,116.43) .. (141.06,91.05) ;
\draw    (130.92,89.15) .. controls (124.4,99.3) and (135.26,103.74) .. (141.06,91.05) ;
\draw    (149.62,88.51) .. controls (145.22,107.28) and (158.56,111.97) .. (167.87,91.25) ;
\draw    (154.66,89.69) .. controls (148.99,95.94) and (157.8,106.89) .. (167.87,91.25) ;
\draw    (159.06,90.08) .. controls (153.4,96.33) and (162.83,99.07) .. (167.87,91.25) ;
\draw    (170.14,90.04) .. controls (177.85,105.44) and (189.15,105.52) .. (183.37,86.94) ;
\draw    (174.4,90.46) .. controls (174.39,97.07) and (185.87,102.03) .. (183.37,86.94) ;
\draw    (177.31,89.19) .. controls (177.31,95.8) and (184.62,94.49) .. (183.37,86.94) ;
\draw    (188.25,84.91) .. controls (195.04,100.39) and (204.99,100.46) .. (199.9,81.79) ;
\draw    (192,85.33) .. controls (191.99,91.98) and (202.1,96.96) .. (199.9,81.79) ;
\draw  [dash pattern={on 0.84pt off 2.51pt}]  (204.92,85.6) -- (212.64,81.16) ;
\draw  [fill={rgb, 255:red, 155; green, 155; blue, 155 }  ,fill opacity=0.15 ] (349.58,61.05) .. controls (349.58,44.48) and (378.78,31.04) .. (414.79,31.04) .. controls (450.81,31.04) and (480,44.48) .. (480,61.05) .. controls (480,77.62) and (450.81,91.05) .. (414.79,91.05) .. controls (378.78,91.05) and (349.58,77.62) .. (349.58,61.05) -- cycle ;
\draw    (377.98,37.13) .. controls (379.43,-4.1) and (424.21,-0.29) .. (427.11,30.79) ;
\draw    (387.4,33.2) .. controls (390.3,14.3) and (420.01,1.61) .. (427.11,30.79) ;
\draw    (397.55,32.57) .. controls (407.69,7.32) and (425.08,19.37) .. (427.11,30.79) ;
\draw    (363.49,79.63) .. controls (302.26,88.04) and (317.29,66.84) .. (357.99,46.53) ;
\draw    (356.25,73.93) .. controls (318.2,74.08) and (330.33,64.94) .. (357.99,46.53) ;
\draw  [dash pattern={on 0.84pt off 2.51pt}]  (414.94,27.62) -- (423.85,27.24) ;
\draw    (406.24,30.79) .. controls (409.86,16.84) and (427.25,21.91) .. (427.11,30.79) ;
\draw    (471.02,46.52) .. controls (469.4,5.29) and (431.5,0.98) .. (428.27,32.06) ;
\draw    (465.22,41.45) .. controls (461.99,22.54) and (436.19,2.88) .. (428.27,32.06) ;
\draw    (459.42,38.91) .. controls (448.11,13.66) and (430.53,20.64) .. (428.27,32.06) ;
\draw    (450,35.74) .. controls (445.96,21.78) and (428.11,23.18) .. (428.27,32.06) ;
\draw  [dash pattern={on 0.84pt off 2.51pt}]  (431.89,30.03) -- (445.66,32.57) ;
\draw  [dash pattern={on 0.84pt off 2.51pt}]  (348.28,56.16) -- (345.38,68.22) ;
\draw    (380.16,86.61) .. controls (375.09,117.06) and (390.45,124.67) .. (401.17,91.05) ;
\draw    (385.95,88.51) .. controls (379.43,98.66) and (389.58,116.43) .. (401.17,91.05) ;
\draw    (391.03,89.15) .. controls (384.51,99.3) and (395.37,103.74) .. (401.17,91.05) ;
\draw    (409.73,88.51) .. controls (405.33,107.28) and (418.67,111.97) .. (427.98,91.25) ;
\draw    (414.77,89.69) .. controls (409.1,95.94) and (417.91,106.89) .. (427.98,91.25) ;
\draw    (419.17,90.08) .. controls (413.51,96.33) and (422.94,99.07) .. (427.98,91.25) ;
\draw    (430.25,90.04) .. controls (437.96,105.44) and (449.26,105.52) .. (443.48,86.94) ;
\draw    (434.51,90.46) .. controls (434.5,97.07) and (445.98,102.03) .. (443.48,86.94) ;
\draw    (437.42,89.19) .. controls (437.42,95.8) and (444.73,94.49) .. (443.48,86.94) ;
\draw    (448.36,84.91) .. controls (455.15,100.39) and (465.1,100.46) .. (460.01,81.79) ;
\draw    (452.11,85.33) .. controls (452.11,91.98) and (462.21,96.96) .. (460.01,81.79) ;
\draw  [dash pattern={on 0.84pt off 2.51pt}]  (465.03,85.6) -- (472.75,81.16) ;
\draw [color={rgb, 255:red, 0; green, 0; blue, 0 }  ,draw opacity=1 ][line width=1.5]    (377.12,84.33) .. controls (245.25,101.46) and (316.98,62.13) .. (356.1,46.9) ;

\draw (137.85,107.07) node [anchor=north west][inner sep=0.75pt]    {$M{_{F'}^{*}}$};
\draw (400.34,110.52) node [anchor=north west][inner sep=0.75pt]    {$M_{F}^{*}$};
\draw (247,32.4) node [anchor=north west][inner sep=0.75pt]    {$f_{F}{}_{'}{}_{F}$};

\end{tikzpicture}

    \caption{Consider $F',F\in \F$ such that $F'$ contains two vertices from $V$ and $F$ contains one vertex from $V$. 
    The figure describes the kettlebell spaces $M_{F'}^*$ and $M_F^*,$ along with the bonding map $f_{F'F}$ which collapses the dark gray vertex space in $M_{F'}^*$ to the thick peripheral arc in $M_F^*$. }
    \label{fig:kettlebell}
\end{figure}
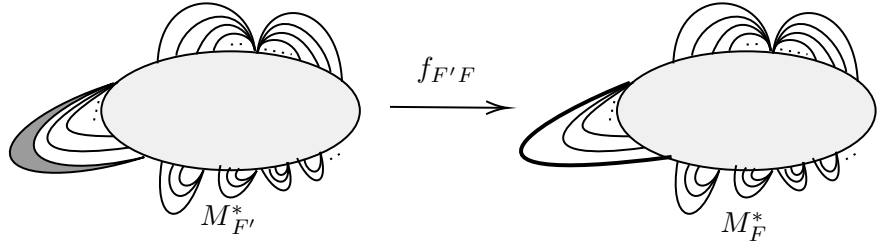

\begin{lem}
\label{bondingauglip}
If $F'$ augments $F$ then the bonding map $f_{F'F}$ in Definition~\ref{bondingaug} is $1$--Lipschitz.
\end{lem}
\begin{proof}
Consider distinct $x,y\in M_{F'}^*$. We want to prove that for any $x,y\in M_F$ we have $d(f_{F'F}(x),f_{F'F}(y))\leq d(x,y)$. Recall that $M_F^*(e)\cup M_v^*(e)=M_{F'}^*$. The case when both $x,y\in M_F^*(e)$ is obvious and the case when $x,y\in M_v^*(e)$ follows from Lemma~\ref{distdec}. Suppose $x\in M_F^*(e)$ and $y\in M_v^*(e)$. By the definition of the quotient metric in Proposition \ref{semimeteff} we either have $d(x,y)=d(x,a)+d(a,y)$ or $d(x,y)=d(x,b)+d(b,y)$ in $M_{F'}^*$. We assume the former and the other case follows similarly. Since $u$ is $1$--Lipschitz,
$$d(f_{F'F}(x),f_{F'F}(y))=d(x,u(y))\leq d(x,a)+d(u(a),u(y))\leq  d(x,a)+d(a,y)$$
which proves that $f_{F'F}$ is $1$--Lipschitz.
\end{proof}

We now define the bonding map in general.
\begin{defn}
\label{bondinggen}
Consider $F,F'\in \mathcal{F}$ such that $F\subset F'$. Then there is an \emph{augmenting chain} from $F'$ to $F$ which is a finite sequence of trees $F=F_1,F_2, \ldots, F_n=F'$ such that $F_{i+1}$ augments $F_i$ for $i=1,2\ldots,n-1$. Then define the \emph{bonding map}
$$f_{F'F}:=f_{F_2 F_1}\circ f_{F_3 F_2} \circ \cdots \circ f_{F_{n-1} F_{n-2}}\circ f_{F_n F_{n-1}}.$$
\end{defn}

Apriori, it may seem that the above definition depends on the choice of the augmenting chain but we prove that this bonding map is well defined.
\begin{lem}
\label{bondingwelldef}
Consider $F,F', F''\in \mathcal{F}$ such that $F\subset F'$. Then the bonding map $f_{F'F}$ in Definition~\ref{bondinggen} is well defined and $1$--Lipschitz. Moreover, $f_{F''F}=f_{F''F'}\circ f_{F'F}$ whenever $F\subset F'\subset F''$. 
\end{lem}
\begin{proof}
Since $f_{F'F}$ is a finite composition of $1$--Lipschitz maps, $f_{F'F}$ is $1$--Lipschitz. To check that the map is well defined we first consider the case when $F'$ has two more edges than $F$.
If we have two augmenting chains $F,F_1,F'$ and $F,\Tilde{F_1},F'$ between $F$ and $F'$ 
then note that $\Image(f_{F_1 F})$ and $\Image(f_{\Tilde{F_1}F})$ can intersect on at most one point, on which both maps agree. Therefore, both augmenting chains give the same definition for $f_{F'F}$.

For the general case, note that the graph $\text{cl}(F'\setminus F)$ is a forest with each tree rooted at a vertex in $F$. The descendancy in the trees determines a partial order and the choice to extend this to a total order is equivalent to choosing an augmenting chain from $F'$ to $F$. Given two total orders extending the same partial order, we can get from one to another by a finite sequence of transpositions of adjacent elements as explained in the previous paragraph.

To see that $f_{F''F}=f_{F''F'}\circ f_{F'F}$ it suffices to consider an augmenting chain from $F$ to $F''$ containing $F'$. 
\end{proof}

\begin{defn}
\label{invsys}
Define the inverse system $\mathcal{S}_{\Theta}$ as
$$\mathcal{S}_{\Theta}=(\{M_F^* \mid F\in \F\},\{f_{F'F} \mid F\subset F'\}).$$
\end{defn}

The inverse system $\mathcal{S}_{\Theta}$
is well defined since the bonding maps commute by Lemma~\ref{bondinggen}.
Define $M_{\Theta}:=\varprojlim \mathcal{S}_{\Theta}$ as the inverse limit of this system which is compact as the factor spaces are kettlebell spaces which are compact. We now construct a map from $M_T$ in Setup~\ref{compsetup} to $M_{\Theta}$.

\begin{lem}
\label{gam}
Consider some $F\in \F$. Then there is a $1$--Lipschitz map $\gamma_F\colon M_T\rightarrow M_F^*$ with the following properties.
\begin{enumerate}
    \item 
    \label{gam: gamF} For any $F'\in \F$ such that $F\subset F'$ we have $\gamma_F\big|_{M_{F'}}=f_{F'F}\big|_{M_{F'}}$.
    \item
    \label{gam: gam props}The maps $\{\gamma_F\}_{F\in \F}$ induce a map $\gamma\colon M_T\rightarrow M_{\Theta}$ that is continuous and injective.
\end{enumerate}
\end{lem}
\begin{proof}
Consider $x\in M_T$. Then there is $F'\in \F$ such that $x\in M_{F'}$ and $F\subset F'$. Define $\gamma_F(x)=f_{F'F}(x)$. 
We show that this is independent of the choice of $F'$. Consider $F_1,F_2\in F$ such that $x\in M_{F_1},M_{F_2}$ and $F\subset F_1,F\subset F_2$. Then there is $F_3$ such that $F_1,F_2\subset F_3$. Using Lemma~\ref{bondingwelldef} and the fact that $f_{F_3F_1}$ is the identity on $M^*_{F_1}$, we get $f_{F_3F}(x)=f_{F_1F}(x)$. Similarly, we get $f_{F_3F}(x)=f_{F_2F}(x)$. Thus $\gamma_F$ is well-defined.

To see that $\gamma_F$ is $1$--Lipschitz, consider $x,y\in M_T$. Then there is $F'\supset F$ such that $x,y\in M_{F'}$ and thus $\gamma_F(x)=f_{F'F}(x)$ and $\gamma_F(y)=f_{F'F}(y)$. Since $f_{F'F}$ is $1$--Lipschitz by Lemma~\ref{bondingauglip} we get that $d(\gamma_F(x),\gamma_F(y))\leq d(x,y)$ which is what we wanted to prove.

To check that $\gamma$ is well defined, consider $F_1,F_2\in \F$ such that $F_1\subset F_2$. We need to prove that $f_{F_2F_1}(\gamma_{F_2}(x))=\gamma_{F_1}(x)$ for all $x\in M_T$. Consider some $F_3$ such that $F_1,F_2\subset F_3$ and $x\in M_{F_3}$. Then since $\gamma_{F_1}=f_{F_3F_1}(x)=f_{F_2F_1}(f_{F_3F_2}(x))$ and $\gamma_{F_2}=f_{F_3F_2}(x)$, we get what we wanted to prove. 

Since all $\gamma_F$'s are $1$--Lipschitz and thus continuous, we  get that $\gamma$ is also continuous. To see that $\gamma$ is injective, consider distinct $x,y\in M_T$. Then there is $F\in F$ such that $x,y\in M_F$. Since $f_{FF}=id_{M_F^*}$ we get that $\gamma_F(x)=x$ and $\gamma_F(y)=y$ which proves injectivity.
\end{proof}

Recall that an end is redundant if eventually the cut pairs corresponding to the end intersect at the same point in $M_T$. We now see that such ends do not add any extra points to the inverse limit.
We first define a map from the ends of $T$ to $M_{\Theta}$. 

\begin{lem}
\label{boundgam}
For each $F\in \F$, there is a map $\boundary\gamma_F\colon \boundary T\rightarrow M_F^*$ with the following properties.
\begin{enumerate}
\item 
\label{boundgam: deffac}
For every $x\in \boundary T$ and every associated sequence $(p_n)$ for $x$, we have that $\gamma_F(p_n)\rightarrow \boundary\gamma_F(x)$.
\item
\label{boundgam: def}The maps $\{\boundary\gamma_F\}_{F\in \F}$ induce a map $\boundary\gamma\colon\boundary T\rightarrow M_{\Theta}$ such that $\gamma(p_n)\rightarrow \boundary\gamma(x)$
for every $x\in \boundary T$ and every associated sequence $(p_n)$ for $x$.
\item 
\label{boundgam: nonredint}
For every $x\in \boundary_0T$ and every $F\in \F$, we must have that $\boundary\gamma_F(x)$ is in the interior of a peripheral arc in $M_F^*$.
\item
\label{boundgam: inj}Restricted to $\boundary_0T$, the map $\boundary\gamma$ is injective.
\end{enumerate}
\end{lem}
\begin{proof}
Consider an associated sequence $(p_n)$ for some $x\in \boundary T$ and some $F\in \F$. Then consider the sequence $(\gamma_F(p_n))$. Since $\gamma_F$ is $1$--Lipschitz and thus uniformly continuous by Lemma~\ref{boundgam}, we get that $(\gamma_F(p_n))$ is Cauchy. Since $M_F^*$ is compact and thus complete,
$(\gamma_F(p_n))$ is convergent. Define $\boundary\gamma_F(x)$ to be the limit of $(\gamma_F(p_n))$. By Lemma~\ref{assocseqprops} and since $\gamma_F$ is $1$--Lipschitz, we get that this definition is independent of the choice of $(p_n)$.

Now define $\boundary\gamma(x)=(\boundary\gamma_F(x))_{F\in \F}$ for all $x\in M_T$. To see that this is well defined consider $F\subset F'$ and some $x\in M_T$ along with an associated sequence $(p_n)$ for $x$. Then we have $\gamma_F(p_n)\rightarrow \boundary\gamma_F(x)$ and $\gamma_{F'}(p_n)\rightarrow \boundary\gamma_{F'}(x)$. Since $\gamma$ is well defined by Lemma~\ref{gam}, and since $f_{F'F}$ is continuous we get that $f_{F'F}(\boundary\gamma_{F'}(x))=\boundary\gamma_F(x)$ which proves that $\boundary\gamma$ is well defined. Since $\gamma_F(p_n)\to \boundary\gamma_F(x)$ for all $F\in \F$, we get (\ref{boundgam: def}).

Consider $x\in \boundary_0 T$ represented by a ray $(v_1,w_1,v_2,w_2\ldots)$ such that $v_1\in \vertices(F)\cap V$ and $v_2\notin V(F)$. Also consider an associated sequence $(p_n)$ for $x$ such that $p_n\in M_{w_n}$. Let $F'=F\cup (v_1,w_1)$. Since $M_F^*=M_{F'}^*$ even if $F\subsetneq F'$, we assume $F$ contains $e_1=(v_1,w_1)$. Let $\ell_1$ be the peripheral arc in $M_F^*$ corresponding to $e_1$. Then note that $\gamma_F(p_n)\in \ell_1$ for all $n$. Since $x$ is non-redundant there is some $k$ such that $M_{w_n}$ does not intersect the endpoints of $\ell_1$ for all $n\geq k$.

Let $e_k=(v_k,w_k)$ and $F_k=F\cup(v_1,w_1,\ldots,v_k,w_k)$. Also let $\ell_k$ be the peripheral arc in $M_{F_k}^*$ corresponding to $e_k$. Then $\gamma_{F_k}(p_n)\in \ell_k$ for all $n\geq k$. Moreover, by the assumption on $k$ we must have that $f_{F_kF}(\ell_k)$ is in the interior of the peripheral arc $\ell_1$. Since $\gamma_F=f_{F_kF}(\gamma_{F_k})$ we get that $p_n$ is in the interior of $\ell_1$ for all $n\geq k$. Since $\gamma_F(p_n)\rightarrow \boundary\gamma_F(x)$ we have that $\boundary\gamma_F(x)$ is in the interior of $\ell_1$ which proves (\ref{boundgam: nonredint}).

Now consider distinct $x,x'\in \boundary_0T$ with representative rays $(v_1,w_1,v_2,w_2,\ldots)$ and $(v'_1,w'_1,v'_2,w'_2,\ldots)$ respectively such that the rays only intersect at $v_1=v'_1=v$. Consider the single vertex subtree $v\in \F$ and $e=(v_1,w_1),e'=(v'_1,w'_1)$ incident on $v$. Let $\ell,\ell'$ be peripheral arcs in $M_v^*$ corresponding to $e,e'$ respectively. Following the argument for (\ref{boundgam: nonredint}) we get $\boundary\gamma_v(x)$ is in the interior of $\ell$ and $\boundary\gamma_v(x')$ is in the interior of $\ell'$ which proves (\ref{boundgam: inj}).
\end{proof}

We now give an alternate way to describe $\boundary\gamma_F$ for $F\in \F$.

\begin{lem}
\label{boundgam alt def}
Consider $F\in \F$ and an end $x\in \boundary T$ represented by the ray $(v_1,w_1,v_2,w_2,\ldots)$ such that $v_1\in V(F)$ but $v_2\notin V(F)$. Let $F_1=F\cup (v_1,w_1)$ and $F_n=F_{n-1}\cup (w_{n-1},v_n,w_n)$. Let $\ell_n$ be the peripheral arc in $M_{F_n}^*$ corresponding to $(v_n,w_n)$ and $A_n=f_{F_nF}(\ell_n)$. Then $\bigcap_{n\geq 1}A_n=\{\boundary\gamma_F(x)\}$.
\end{lem}
\begin{proof}
Consider the associated sequence $(p_n)$ for $x$ such that $p_n$ is one of the endpoints of $\ell_n$. Then by Lemma~\ref{gam}(\ref{gam: gamF}) we get that $\gamma_F(p_n)\in A_n$ for all $n$. Since $A_n$ is compact and thus closed for all $n$, by Lemma~\ref{boundgam}(\ref{boundgam: def}) we get that $\boundary\gamma_F(x)\in A_n$ for all $n$. Since $\diam(\ell_n)\rightarrow 0$ and since the maps $f_{F_nF}$ are $1$--Lipschitz by Lemma~\ref{bondingwelldef} we get that $\diam(A_n)\rightarrow 0$. The result then follows from the Cantor intersection theorem.
\end{proof}

The following result gives a converse of Lemma~\ref{boundgam}(\ref{boundgam: nonredint}).

\begin{lem}
\label{interiorthreads}
Consider a point $x=(x_F)_{F\in \F}\in M_{\Theta}$ such that for all $F\in \F$ we have $x_F$ is in the interior of a peripheral arc in $M_F^*$. Then there is a unique $y\in \boundary_0 T$ such that $\boundary\gamma(y)=x$.
\end{lem}
\begin{proof}
Fix some $F^0\in \F$. Then we construct a ray $(v_1^0,w_1^0,v_2^0,w_2^0,\ldots)$ representing $y$ such that $v_i^0\in V, w_i^0\in W$ and $v_1^0\in V(F^0),v_2^0\notin V(F^0)$. Then $x_{F^0}$ is in some peripheral arc $\ell^0_1$ in $M_{F^0}^*$. This arc corresponds to some edge $e_1^0$ neighboring $v_1^0$. We define $w_1^0\in W$ as the other vertex of $e_1^0$.
Assuming we have already defined $v_1^0,w_1^0,\ldots, v_i^0,w_i^0$, we define $v_{i+1}^0$ and $w_{i+1}^0$ as follows. Let $F_i^0$ be the finite tree consisting of $w_i^0$ along with all it's neighboring edges. Then $x_{F_i^0}$ is in some peripheral arc $\ell^0_i$ corresponding to some edge $e^0_i$. Define $v^0_{i+1}$ and $w^0_{i+1}$ as the vertices of $e^0_i$.

Now consider any arbitrary $F\in \F$. We now prove that $\boundary\gamma_F(y)=x_F$. Consider a ray $(v_1,w_1,v_2,w_2,\ldots)$ representing $y$ such that $v_1\in V(F)$ but $v_2\notin V(F)$. Then there is $m$ such that $(v_{m},w_m,v_{m+1},w_{m+1}\ldots)$ is contained in $(v_1^0,w_1^0,v_2^0,w_2^0,\ldots)$. Let $\ell_n$ and $A_n$ and $F_n$ be as in Lemma~\ref{boundgam alt def}. Then we have that $x_F=f_{F_mF}(x_{F_m})\in \ell_1$ as $x_{F_m}$ is in the interior of a peripheral arc in $M_{F_m}^*$. Moreover, $x_F=f_{F_nF}(x_{F_n})\in A_n$ for all $n\geq m$. Then by Lemma~\ref{boundgam alt def} we get that $x_F=\boundary\gamma_F(y)$.

Now note that $y$ is non-redundant, otherwise there is an associated sequence $p_n$ for $y$ with a point $p\in M_T$ and $N>0$ such that for all $n>N$, we have $p_n=p$. Then there is $F\in F$ such that $p\in M_F$. Applying Lemma~\ref{boundgam}(\ref{boundgam: deffac}) we get that $\boundary\gamma_F(y)=x_F=p$ which is a contradiction as $x_F=p$ is not in the interior of a peripheral arc. The uniqueness of $y$ then follows from Lemma~\ref{boundgam}(\ref{boundgam: inj}).
\end{proof}

In \cite{Swiatkowski20}, the inverse limit constructed is in bijection with $M_T\sqcup \boundary T$. However, since we are allowing the edge spaces to intersect in the component spaces, while the inverse limit $M_{\Theta}$ does still contain $M_T$, it is not in bijection with $M_T\sqcup \boundary T$. The inverse limit is in bijection with $M_T\sqcup \boundary_0T$ as shown by the following theorem.
Recall that by Lemma~\ref{boundalpha} and Proposition~\ref{completion}, we can identify $\bar{M}_T\setminus M_T$ with $\boundary_0T$.

\begin{thm}
\label{completion is inverse limit}
The map $\gamma\colon M_T\rightarrow M_{\Theta}$ has a unique extension $\bar{\gamma}\colon\bar{M}_T\rightarrow M_{\Theta}$ which is a homeomorphism. Moreover, $\bar{\gamma}\big|_{\boundary_0T}=\boundary\gamma\big|_{\boundary_0T}$.
\end{thm}

\begin{proof}
From Lemma~\ref{gam} we have the $1$--Lipschitz and thus uniformly continuous map $\gamma_F\colon M_T\rightarrow M_F^*$ for all $F\in \F$. Since $\gamma_F$ is uniformly continuous, there is a unique extension $\bar{\gamma}_F\colon \bar{M}_T\rightarrow M_{\Theta}$.
Now consider some $x\in \bar{M}_T\setminus M_T=\boundary_0T$. Then by Lemma~\ref{boundalpha}, Proposition~\ref{completion} and Lemma~\ref{boundgam} we get that $\bar{\gamma}_F\big|_{\boundary_0T}=\boundary\gamma_F\big|_{\boundary_0T}$. 
We now define $\bar{\gamma}(x)=(\bar{\gamma}_F)_{F\in \F}$ which is well defined as $\gamma$ and $\boundary\gamma$ are well defined by Lemma~\ref{gam} and Lemma~\ref{boundgam}. Since $\bar\Psi$ is a continuous map between compact Hausdorff spaces we only need to prove that it is bijective.

We first prove that $\bar{\gamma}$ is injective. Consider $x,y\in \bar{M}_T$. If $x,y\in M_T$ then we get injectivity from Lemma~\ref{gam}. If $x,y\in \boundary_0T$ then we get injectivity from Lemma~\ref{boundgam}(\ref{boundgam: inj}). If $x\in M_T$ and $y\in \boundary_0T$ then consider some $F\in \F$ such that $x\in M_F$. Then note that $\bar{\gamma}_F(y)$ is in the interior of a peripheral arc by Lemma~\ref{boundgam}(\ref{boundgam: nonredint}) while $\bar{\gamma}_F(x)=\gamma_F(x)=x$ is not.

Now we prove that $\bar{\gamma}$ is surjective. Consider some $x=(x_F)_{F\in \F}\in M_{\Theta}$. If there is $F_0\in \F$ such that $x_{F_0}\in M_{F_0}\subset M_{F_0}^*$, let $y=x_{F_0}\in M_T$. Then $\gamma(y)=x$. Otherwise we must have that $x_F$ is in the interior of a peripheral arc in $M_F^*$ for all $F\in \F$. Then from Lemma~\ref{interiorthreads} we get that there is $y\in \boundary_0T$ such that $\bar{\gamma}(y)=\boundary\gamma(y)=x$.
\end{proof}


\section{Combining Bowditch boundaries along cut pairs}
\label{sec:homeo}

In this section, we explore the topological properties of the Bowditch boundary $\boundary(G,\P)$ for the relatively hyperbolic pair $(G,\P)$ we obtained in Theorem~\ref{combination}. We define a tree system $\Theta_G$ (Definition~\ref{ThetaG}) with vertex spaces made up of boundaries of the vertex stabilizers of the tree $T$ in Theorem~\ref{combination}. We then prove that the boundary $\boundary(G,\P)$ is homeomorphic to the the completion $\bar{M}_T$ of the total space $M_T$ associated to $\Theta_G$ (Theorem~\ref{boundarytreesys}).
First we prove that the orbit of a pair of parabolic points is a null collection.

\begin{prop}
\label{parabolicnull}
Consider a relatively hyperbolic pair $(G,\P)$ and two parabolic points $a,b\in \partial(G,\P)$. Then $\mathcal{C}=\{g\{a,b\}\mid g\in G \}$ forms a null collection in $\partial(G,\P)$.
\end{prop}
\begin{proof}
Let $K$ be a $(G,\P)$ graph. Then $a,b\in \vertices(K)$. Moreover, by Lemma~\ref{gattach} we can assume that there is an edge $e$ joining $a$ and $b$ in $K$. Let $d$ be a metric on $\partial(G,\P)$. 

Assume for contradiction that $\mathcal{C}$ is not a null collection in $\boundary(G,\P)$. Then there exist $\epsilon>0$ and a sequence $(\{g_na,g_nb\})_{n\in \N}$ of distinct members of $\mathcal{C}$ such that $d(a_n,b_n)>\epsilon$ for all $n$ where $a_n=g_n a$ and $b_n=g_n b$. Note that both $(a_n)$ and $(b_n)$ cannot be eventually constant sequences as then $(\{g_na,g_nb\})_{n\in \N}$ does not consist of distinct elements. Without loss of generality, assume that $(b_n)$ is not an eventually constant sequence.
By passing to a subsequence, assume that $a_n\rightarrow a_{\infty}$ for some $a_{\infty}\in \partial(G,\P)$. 

We now define $s\colon\N\rightarrow \N$ as $s(k)=k+1$. Then $s\geq 1_{\N}$. Moreover, there is a finite set $A\subset V(K)$ such that $M'_s(a_{\infty},A)\subset B_{\epsilon/2}(a_{\infty})$. 
Since $1_{\N}\leq s$ we have $M'(a_{\infty},A)\subset M'_s(a_{\infty},A)$.
As $A$ is finite and as $(b_n)$ is not eventually constant, there is an $i$ such that $b_i\notin A$ but $a_i\in M'(a_{\infty},A)$. Since $a_i\in M'(a_{\infty},A)$, there is a geodesic path $\gamma$ from $a_i$ to $a_{\infty}$ which misses $A$. We form an $s$--quasigeodesic path $\gamma'$ by concatenating the edge $g_ie=(a_i,b_i)$ followed by $\gamma$. Thus $b_i\in M'_s(a_{\infty},A)\subset B_{\epsilon/2}(a_{\infty})$ which proves that $d(a_i,b_i)<\epsilon$ which is the required contradiction.
\end{proof}

We now define an object for a $2$--admissible action and a signature.

\begin{defn}
\label{ThetaG}
Consider a group $G$ acting on a tree $T$ with a $2$--admissible action where $G$ and $T$ are as described in Definition~\ref{2admissible}. Then we have relatively hyperbolic pairs $(G_u,\P_u)$ and $(G_u,\P_u)$ graphs $K_u$ for each $u\in V\sqcup W=\vertices(T)$.
Let $\mathcal{S}=\{s_e\}_{e\in \edges(T)}$ be a signature for this action. Then by Theorem~\ref{combination}, we get that $(G,\P)$ is relatively hyperbolic along with a $(G,\P)$ graph $\bar{K}$. Using Proposition~\ref{parabolicnull}, we define a tree system of cut pairs
$\Theta_G$ for $\mathcal{S}$ over $T$ consisting of the following data.
\begin{enumerate}
    \item The vertex spaces $M_u=\Delta K_u$ for all $u\in \vertices(T)$.
    \item \label{ThetaG: injections} For an edge $e=(v,w)$ in $\edges(T)$, for $x\in \Lambda_{{w}}=\Delta K_w=M_w$, let $i_e(x)=s_e(x)\in \Lambda_v\subset \Delta K_v=M_v$ be injections.
\end{enumerate}
\end{defn}

\begin{lem}
\label{Psi}
Consider $\Theta_G$ as in Definition~\ref{ThetaG} and consider the metric space $(M_T,d)$ as described in Setup~\ref{compsetup} with $\Theta=\Theta_G$. We then have an injective continuous map $\Psi\colon M_T\rightarrow \Delta\bar{K}$ with the following property. For each $v\in V$ we have that $\Psi\big|_{M_v}$ is the identity on $M_v=\Delta K_v$.
\end{lem}
\begin{proof}
Let $\Delta_{\sqcup},{\equiv}$ and $\sigma$ be as in Lemma~\ref{intersection of Delta Kv's} and let $M_{\sqcup}$ be as in Setup~\ref{compsetup}. Then as $M_u=\Delta K_u$ for all $u\in \vertices(T)$, there is an obvious bijection $\eta_{\sqcup}\colon M_{\sqcup}\to \Delta_{\sqcup}$. Moreover, by the definition of $\equiv$ and $\sim$, we have $x\sim y$ if and only if $\eta_{\sqcup}(x)\equiv \eta_{\sqcup}(y)$ for all $x,y\in M_{\sqcup}$. Therefore we have a bijection $\eta\colon M_T\to \Delta_{\sqcup}/{\equiv}$. Defining $\Psi=\sigma\circ \eta$ we get the required injective map.

We now show that $\Psi$ is continuous. For some $x\in M_T$ consider a basic open set $M(\Psi(x),C)$ in $\Delta\bar{K}$. We now find an $r$ such that $\Psi(B_r(x))\subset M(\Psi(x),C)$.

Let $C_1\subset C$ be the subset consisting of $c$ such that $c$ and $\Psi(x)$ are not in the same $\Delta K_v$ in $\Delta \bar{K}$. Then for any $c\in C_1$ there is some $w_c\in W$ such that any geodesic joining $\Psi(x)$ and $c$ in $\bar{K}$ intersects $\Lambda_{w_c}$.

Define $C_2=C\setminus C_1$ and consider some $c\in C_2$. Let $V_c$ be the finite collection consisting of $v\in V_c$ such that $\Psi(x)$ and $c$ are in the same $\Delta{K}_v$ in $\Delta\bar{K}$. Let $s\colon\N\rightarrow \N$ be defined as $s(k)=k+1$. Since $\Psi\big|_{M_{v}}$ is an embedding there is $r_v$ such that $\Psi(B_{r_v}(x))\subset M_s(\Psi(x),C\cap\Lambda_{v})$ in $\Delta K_{v}$. Now define
$$r:=\min\Biggl[\{d(x,M_{w_c})\mid c\in C_1\}\cup \bigcup_{c\in C_2}\{r_v\mid v\in V_c\}\Biggr].$$
Then we claim that $\Psi(B_r(x))\subset M(\Psi(x),C)$. Suppose we have $y\in M_T$ and a geodesic $\alpha$ in $\bar{K}$ joining $\Psi(x)$ and $\Psi(y)$ which goes through some $c\in C$. We prove that $y\notin B_r(x)$.

First consider the case when $c\in C_1$. Note that the pair $M_{w_c}$ separates $x$ and $y$ in $M_T$. Then by the definition of the quotient metric $d$ in Proposition~\ref{semimeteff} there is $x_c\in M_{w_c}=\Lambda_{w_c}$ such that $d(x,y)=d(x,x_c)+d(x_c,y)$. Thus $y\notin B_r(x)$ as $d(x,x_c)\geq r$, which proves the claim in the first case.

Now consider the second case when $c\in C_2$. Let $y\in M_{v_y}$ for some $v_y\in V$. Then choose $v\in V_c$ closest to $v_y$ in $T$. If $v=v_y$ then $y\notin B_r(x)$ as $\Psi(B_{r_v}(x))\subset M_s(\Psi(x),C\cap\Lambda_{v})$ in $M_{v}$. Now assume $v\neq v_y$. Then there is some $w$ adjacent to $v$ in $T$ such that $x$ and $y$ are separated by $M_{w}=\Lambda_w=\{a_w,b_w\}$ in $M_T$. If $\Psi(B_r(x))\cap \Lambda_{w}=\emptyset$ then using $d(x,M_{w_c})<r$ we can proceed as in the above case to get $y\notin B_r(x)$ which proves the claim in the second case. We now prove that $\Psi(B_r(x))\cap \Lambda_{w}=\emptyset$ to complete the proof.

If $\alpha$ crosses both $a$ and $b$ then we have geodesics from $a_w$ and $b_w$ to $x$ which go through $C\cap \Lambda_{v}$.
Thus $\Psi(B_r(x))\cap \Lambda_{w}=\emptyset$ as $\Psi(B_{r}(x))\subset M_s(\Psi(x),C\cap\Lambda_{v})$. Now assume without loss of generality that $\alpha$ only crosses $a$. Let $\alpha_1=\alpha\cap K_v$. Then $a\notin \Psi(B_r(x))$ as $\Psi(B_{r}(x))\subset M_s(\Psi(x),C\cap\Lambda_{v})$. Let $e$ be the edge between $a$ and $b$ in $\bar{K}$. Then concatenating $e$ and $\alpha$ we get a $s$--geodesic from $b$ to $x$ which goes through $C\cap \Lambda_v$. Thus $b_w\notin \Psi(B_r(x))$ which proves the claim in the second case.
\end{proof}

The following theorem helps us to extend the map $\Psi$ to $\bar{M}_T$.
\begin{thm}[Theorem~I.8.1, \cite{Bourbaki}]
\label{bourbaki}
Consider a space $X$, a dense subset $A\subset X$, and a continuous map $\phi\colon A\rightarrow Y$ such that $Y$ is a metrizable space.
Suppose for each $x\in X$, there is $y\in Y$ satisfying the following. For each neighborhood $V\ni y$, there is a neighborhood $U\ni x$ such that $\phi(U\cap A)\subset V$. Then $\phi$ extends uniquely to a continuous map $\bar{\phi}:X\rightarrow Y$.
\end{thm}

\begin{lem}
\label{Psi ext}
Let $\Psi$ be as in Lemma~\ref{Psi}.
There is an injection $\boundary\Psi \colon \bar{M}_T\setminus M_T\to \Delta \bar{K}\setminus \Image(\Psi)$ satisfying the following.
For each $x\in \bar{M}_T\setminus M_T$ and each neighborhood $V\ni \boundary\Psi(x)$, there is a neighborhood $U\ni x$ such that $\Psi(U\cap M_T)\subset V$.
\end{lem}
\begin{proof}
Recall from Proposition~\ref{completion} that $\bar{M}_T\setminus M_T=\boundary\alpha(\boundary_0T)$. Then define $\boundary\Psi=\boundary\psi\circ (\boundary\alpha)^{-1}$ where $\boundary\psi$ and $\sigma$ is as in Proposition~\ref{non-red to bound k bar}. As $\boundary\psi$ is injective by Proposition~\ref{non-red to bound k bar}, we get that $\boundary\Psi$ is injective as well.

Let $y=\boundary\Psi(x)$. Then we can assume that $V=M(y,A)$ for some finite subset $A\subset \vertices(\bar{K})$. By Lemma~\ref{neighborhoods of non-redundant ends} there is $w\in W$ such that $M(y,\Lambda_w)\subset M(y,A)$. Let $S_1,S_2,\ldots,S_k$ be the connected components of $T\setminus\{w\}$. Choose $i$ such that $x\in \bar{M}_{S_i}$. Then since $\bigcup_{j\neq i}\bar{M}_{S_j}$ is closed in $\bar{M}_T$, we have that its complement $U:=\bar{M}_{S_i}\setminus M_w$ is open in $\bar{M}_T$. Therefore, $U:=(\bar{M}_{S_i}\setminus M_w)\ni x$. We now prove that $\Psi(U\cap M_T)\subset M(y,A)$.

Suppose there is $z\in U\cap M_T$ with a geodesic $\alpha$ from $\Psi(z)$ to $y$ in $\bar{K}$ passing through some $a\in A$. Suppose $z\in M_v$ and $a\in M_{v'}$ where $v,v'\in V$ are chosen such that $v$ and $v'$ are as close as possible in $T$. Suppose $v\in \vertices(S_i)$. (Otherwise $z\notin U$.)  Note that $v'\notin \vertices(S_i)$ by the choice of $w$ in Lemma~\ref{neighborhoods of non-redundant ends}. Therefore $\alpha$ consists of a geodesic from $\Psi(z)$ to $a$ passing through $\Lambda_w$ followed by a geodesic from $a$ to $x$ passing through $\Lambda_w$ again. However $\alpha$ being a geodesic implies $\Psi(z),a\in \Lambda_w$. But since $U\cap\Lambda_w=\emptyset$ we get $z\notin U$.
\end{proof}

We are now ready to prove that the main theorem for this section.
\begin{thm}
\label{boundarytreesys}
Let $G$ and $\P$ be as in Theorem~\ref{combination} and let $\Theta_G$ be a tree system of cut pairs as described in Definition~\ref{ThetaG}. Also let, $\bar{M}_T$ be the completion of the space $M_T$ as described in Setup~\ref{compsetup}. Then there is a homeomorphism $\bar{\Psi}\colon \bar{M}_T\to \boundary(G,\P)$ such that $\bar{\Psi}\big|_{M_u}$ is the identity for each $u\in \vertices(T)$.
Moreover, the boundary $\boundary(G,\P)$ is also homeomorphic to the inverse limit $M_{\Theta_G}$ as described in Definition~\ref{invsys}.
\end{thm}
\begin{proof}
From Lemmas~\ref{Psi}~and~\ref{Psi ext} above, we have that the map $\Psi$ satisfies the conditions of Theorem~\ref{bourbaki}. Therefore, we have a continuous extension $\bar{\Psi}\colon \bar{M}_T\rightarrow \Delta \bar{K}$ which we now show is an homeomorphism. Recall that $\bar{\Lambda}=\bigcup_{v\in V} \Lambda_v$ from Theorem~\ref{combination}(\ref{combination: Kbar}). Since $\Image(\Psi)$ contains $\Delta_v$ for all $v\in V$, it contains $\Lambda_v$ for all $v\in V$. Therefore $\Image(\Psi)$ contains $\bar{\Lambda}$.  Note that since $\bar{M}_T$ is compact, $\Image(\bar{\Psi})$ is compact and thus closed in $\Delta\bar{K}$. Moreover, since $\bar{\Lambda}=\vertices(\bar{K})$ is dense in $\Delta\bar{K}$ and  $\Image(\bar{\Psi})$ contains $\bar{\Lambda}$ we have that $\Image(\bar{\Psi})=\Delta\bar{K}$ proving that $\bar{\Lambda}$ is surjective.

By Lemma~\ref{Psi} we have that $\bar{\Psi}$ is injective on $M_T$. Since the extension $\bar{\Psi}$ is unique by Theorem~\ref{bourbaki}, from Lemma~\ref{Psi ext} we get $\bar{\Psi}=\boundary\Psi$ on $\bar{M}_T\setminus M_T$. Thus $\Psi$ is an injection on $\bar{M}_T\setminus M_T$ as well. Moreover, from Lemma~\ref{Psi ext} we get that $\bar{\Psi}(M_T)\cap \bar{\Psi}(\bar{M}_T\setminus M_T)=\emptyset$. Therefore $\bar{\Psi}$ is injective.
Since $\bar{\Psi}$ is a continuous bijection between compact Hausdorff spaces, we get that it is a homeomorphism. The last assertion follows from Theorem~\ref{completion is inverse limit}.
\end{proof}
If we impose restrictions on the Bowditch boundaries of the vertex groups, we also get topological information about the Bowditch boundary $\boundary(G,\P)$.

\begin{cor}
\label{bowditchconn}
Given the setup in Theorem~\ref{combination}, if for all $v\in V$, the Bowditch boundaries $\boundary(G_v,\P_v)$ are connected without cut points and cut pairs then $\boundary(G,\P)$ is connected without cut points and only has parabolic cut pairs.

Moreover, if $W_I$ is the collection of all inseparable cut pairs on $\boundary(G,\P)$ and $T_I$ is the inseparable cut pair tree dual to $W_I$, then we have a $G$--equivariant graph isomorphism $\phi_I\colon T\rightarrow T_I$. 
\end{cor}
\begin{proof}
From Proposition~\ref{insepstr}(\ref{insepstr:no cut points}) we get that $\boundary(G,\P)$ is connected without cut points. Also, from Proposition~\ref{insepstr}(\ref{insepstr: insep cutpair}), $W_I=\{\Lambda_w\}_{w\in W}$ as $M_w=\Lambda_w$.

For the second assertion, note that we have a graph isomorphism $\phi_I\colon T\rightarrow T_I$ as described in Proposition~\ref{cutpairtree iso}. Since $\phi(w)=M_w=\Lambda_w$ and $g\Lambda_w=\Lambda_{gw}$, we get that $g\phi(w)=\phi(gw)$ for all $w\in W$ and $g\in G$. Since $\phi_I(v)=\{\Lambda_w\mid w\text{ is adjacent to } v\}$, we get
\begin{align*}
g\phi_I(v)&=\{g\Lambda_w\mid w\text{ is adjacent to } v\}=\{\Lambda_{gw}\mid w\text{ is adjacent to } v\} \\
&=\{\Lambda_w\mid w\text{ is adjacent to } gv\}=\phi(gv)
\end{align*}
for all $v\in V$ and $g\in G$. Thus we get that $\phi_I$ is $G$--equivariant.
\end{proof}

Note that we do not need the Bowditch boundaries of the vertex stabilizers to be connected to for $\boundary(G,\P)$ to be connected without cut points. 
One such instance is seen below, where the boundaries of the vertex stabilizers are disconnected but the boundary $\boundary(G,\P)$ is connected without cut points.

\begin{exmp}
\label{CombExample}
Consider the groups $A=B=C=\Z/3\Z$ and $G=A*B*C$. Then there is $\P$ such that $(G,\P)$ is relatively hyperbolic and $\boundary(G,\P)$ is connected without cut points \cite[Proposition~4.4]{HW}. By Theorem~\ref{decomp}, there is a bipartite tree $T$ on which $G$ acts with a $2$--admissible action and a signature $\mathcal{S}$ for this action such that applying Theorem~\ref{combination} with $\mathcal{S}$, we get the relatively hyperbolic pair $(G,\P)$ again. But, the boundary of each vertex stabilizer of $T$ is a discrete space with three points. 
\end{exmp}

\section{Decomposing along parabolic cut pairs}
\label{sec:decomp}

In this section, we give a converse result to Theorem~\ref{combination}. We explore how to decompose a relatively hyperbolic group with a connected boundary without cut points along inseparable parabolic cut pairs. We regularly use the inseparable cut pair tree and stars as in Definition~\ref{cutpairtree}, and $2$--admissible action as in Definition~\ref{2admissible}. We now state the main result of this section.

\begin{thm}
\label{decomp}
Consider a relatively hyperbolic pair $(G,\P)$ with a connected boundary $\boundary(G,\P)$ without cut points and a non-empty $G$--equivariant set of inseparable parabolic cut pairs $W$ on $\boundary(G,\P)$. Let $T$ be the inseparable cut pair tree dual to $W$ with vertex set $V\sqcup W$, where $V$ is the set of stars as defined in Definition~\ref{cutpairtree}.  Then we have the following.
\begin{enumerate}
    \item $G$ acts on $T$ with a $2$--admissible minimal action as in Definition~\ref{2admissible}.\label{Tminimal}
    \item Let $G_v$ be the stabilizer for $v\in V$. Then there is a indexed family $\P_v$ of subgroups of $G$ such that $(G_v,\P_v)$ is a relatively hyperbolic pair. \label{GvRel}
    \item There is a signature $\mathcal{S}$ of the action of $G$ on $T$ such that the relatively hyperbolic pair $(G,\P)$ is obtained by applying Theorem~\ref{combination} with $\mathcal{S}$.\label{Decomp sign}
    \item There is a tree system $\Theta$ of cut pairs on $T$ with constituent spaces $\boundary(G_v,\P_v)$ such that $\boundary(G,\P)$ is homeomorphic to the inverse limit $M_{\Theta}$ and also to the completion $\bar{M}_T$ of the space $M_T$ defined in Setup~\ref{compsetup}.\label{Decomp treesys}
\end{enumerate}
\end{thm}

We now see some properties about the inseparable cut pair tree.

\begin{lem}
\label{T construction and properties}
Let $T,W, V$ and $G$ be as in Theorem~\ref{decomp}.
Then the following hold.
\begin{enumerate}
    \item \label{T construction and properties: G and P fin gen}
    All members in $\P$ are finitely generated.
    \item \label{T construction and properties: G action on T} $G$ acts minimally on $T$ with finite quotient.
    \item \label{T construction and properties: Gw finite}
    For $u\in \vertices(T)$, let $G_u=\Stab(u)$. Then $|G_w|<\infty$ for $w\in W$.     
\end{enumerate}
\end{lem}
\begin{proof}
From \cite[Proposition~2.29]{Osin}, we have that for any relatively hyperbolic pair, the peripherals subgroups are always finitely generated.
This gives (\ref{T construction and properties: G and P fin gen}). Hruska--Walsh \cite[Lemma~6.3]{HW} show that the action of $G$ on the inseparable cut pair tree of the Bowditch boundary is always minimal.
Dicks--Dunwoody \cite[Proposition~I.4.13]{DD} show that any finitely generated acting minimally on a tree has a finite quotient. This proves (\ref{T construction and properties: G action on T}).
Since every $w\in W$ consists of two parabolic points, we get (\ref{T construction and properties: Gw finite}).
\end{proof}

Recall these facts about peripheral subgroups from Section~\ref{subsec: Bowditch boundaries}. For a parabolic point $p\in \boundary(G,\P)$ if $P=\Stab(p)$ then $P$ acts properly discontinuously on $\boundary(G,\P)\setminus \{p\}$.
Moreover, if $\boundary(G,\P)$ is connected then $P$ is infinite.
We now prove properties about the components in $\boundary(G,\P)\setminus w$ for $w\in W$.

\begin{lem}
\label{Uvw}
Let $T,W$ and $V$ be as in Theorem~\ref{decomp}.
Then for $w\in W$, there is a bijection $\mu_w\colon \Comp(T\setminus w)\to \Comp(\boundary(G,\P)\setminus \{w\})$ satisfying the following.
\begin{enumerate}
\item \label{UvW: direction}
For $w'\neq w\in W$ and $S\in \Comp(T\setminus w)$, we have that $w'\setminus w\subset \mu_w(S)$ if and only if $w'\in \vertices(S)$.
\item \label{Uvw: different directions dont intersect}
Consider $w,w'\in W,S\in \Comp(T\setminus w)$ and $S'\in \Comp(T\setminus w')$
such that $S\cap S'=\emptyset$. Then $\mu_w(S)\cap \mu_{w'}(S')=\emptyset$.
\end{enumerate}
\end{lem}
\begin{proof}
Consider some $w\in W$ and $S\in \Comp(T\setminus w)$. Choose any $w_S\in W\cap S$. Define $\mu_w(S)$ to be the component of $\boundary(G,\P)\setminus w$ containing $w_S\setminus w$. (Such a component exists as $w_S$ is inseparable.) Now consider some other $w'\in W$ such that $w'\in S$. Then $\mu_w(S)$ contains $w'\setminus w$ as well otherwise $w$ is in between $w_S$ and $w'$ in the tree $T$. Now suppose $w'\notin S$. Then $\mu_w(S)$ cannot contain $w'\setminus w$ otherwise $w$ is not in between $w_S$ and $w'$ in $T$. This proves (\ref{UvW: direction}) and that $\mu_w$ is well defined and injective.

Let $U$ be a component of $\boundary(G,\P)\setminus w$. Since parabolic points are dense in the boundary and since $U$ is open, there is a parabolic point $p\in U$. 
Also note that $C=\boundary(G,\P)\setminus U$ is a compact subset in $\boundary(G,\P)\setminus \{p\}$ and $w\subset C$. 
Since $\Stab(p)$ is an infinite group acting properly discontinuously on $\boundary(G,\P)\setminus \{p\}$, there is $g\in \Stab(p)$ such that $gC\cap C=\emptyset$. Therefore, there is $g\in G$ such that $gw\subset U$. Let $S\in \Comp(T\setminus w)$ contain $gw$. Then we must have $U=\mu_w(s)$, which proves that $\mu_w$ is surjective.

For (\ref{Uvw: different directions dont intersect}), we assume that $w\neq w'$ as otherwise we get the result from (\ref{UvW: direction}). Assume for contradiction there is $x\in \mu_w(S)\cap \mu_{w'}(S')$. Then $U=\mu_w(S)\cup \mu_{w'}(S')$ is connected. Since $w\not\subset \mu_{w'}(S')$ by (\ref{UvW: direction}), the space $U$ is connected in $\boundary(G,\P)\setminus\{w\}$. As $\mu_w(S)$ is a connected component of $\boundary(G,\P)\setminus\{w\}$, we have $U=\mu_w(S)$. Now consider some $w''\in S'\cap W$  such that $w''\neq w'$. Then $w''\setminus w'\in \mu_{w'}(S')\subset U=\mu_{w}(S)$, which contradicts (\ref{UvW: direction}). This proves (\ref{Uvw: different directions dont intersect}).
\end{proof}

\begin{prop}
\label{Bvprop}
    Let $T,V, W$ and $G$ be as in Theorem~\ref{decomp}.
    For $v\in V$, and $w\in v$ let $S^w_v\in \Comp(T\setminus w)$ contain $v$. Then define $U^w_v:=\mu_w(S^w_v)$ where $\mu_w$ is as in Lemma~\ref{Uvw}. Now consider the compact subspace $B_v=\bigcap_{w\in v}\text{cl}(U^w_v)$. Then the following hold.
    \begin{enumerate}
        \item For any $w\in W$, we have that $w\subset B_v$ if and only if $w\in v$. \label{pair in B_v}
        \item $\Stab(B_v)=G_v$. \label{StabBv}
        \item For any $g\in G$, we have that $gB_v=B_{gv}$.\label{BvGeq}
        \item Consider any $v'\neq v\in V$. If $x\in B_v\cap B_{v'}$ and $w\in W$ is between $v$ and $v'$ in $T$, then $x\in w$.\label{Bvpath}
    \end{enumerate}
\end{prop}
\begin{proof}
Consider some $w\notin v$. Let $w'\in v$ be closest to $w$ in $T$. Applying Lemma~\ref{Uvw}(\ref{UvW: direction}), we get $w\setminus w'\notin U^{w'}_v$ as $w$ and $v$ are in different components of $T\setminus \{w'\}$. Thus $w\not\subset B_v$. Therefore, $w\subset B_v$ implies that $w\in v$. The definition of $B_v$ gives that $w\in v$ implies $w\subset B_v$. This proves (\ref{pair in B_v}).

Note that $G_v\subset\Stab(G_v)$ by definition. Assume for contradiction $gv\neq v$ for some $g\in \Stab(B_v)$. Since $T$ is a tree we have that $|v\cap gv|\leq 1$. Moreover since $G$ acts minimally on $T$ by Lemma~\ref{T construction and properties}(\ref{T construction and properties: G action on T}), every vertex in $T$ has valence more than one. Therefore both $v$ and $gv$ have at least two elements. Thus there is some $w\in v$ such that $gw\notin v$. From (\ref{pair in B_v}) we get that $w\subset B_v$ but $gw\not\subset B_v$ which implies that $g\notin \Stab(B_v)$. Thus $\Stab(G_v)\subset G_v$ which proves (\ref{StabBv}). As the collection $\{\mu_w\}_{w\in W}$ is $G$--equivariant we get (\ref{BvGeq}).

For (\ref{Bvpath}), choose $w_v\in v$ to be the vertex in $v$ closest to $v'$. Similarly, choose $w_{v'}\in v'$. From $\text{cl}(U^{w_v}_{v})=w_v\cap U^{w_v}_{v}$ and $\text{cl}(U^{w'_{v'}}_{v})=w_{v'}\cap U^{w_{v'}}_{v'}$,
combined with Lemma~\ref{Uvw}(\ref{Uvw: different directions dont intersect}), we get that $x\in w_v\cap w_{v'}$. Now suppose for contradiction there is $w$ between $v$ and $v'$ such that $x\notin w$. Then since $\mu_w$ is surjective by Lemma~\ref{Uvw}, there is $v_w\ni w$ such that $x\in U^w_{v_w}$. Observe that at least one of $v,v'$ are not in the component of $T\setminus \{w\}$ containing $v_w$. First assume $v$ and $v_w$ are not in the same component of $T\setminus\{w\}$. Let $w_v\in v$ be the vertex in $v$ closest to $w$. Then by Lemma~\ref{Uvw}(\ref{UvW: direction}), we get $w_v\setminus w\cap U^w_{w_w}$. As $x\in U^w_{v_w}$, we get $x\notin w_v$ which is a contradiction. The case when $v'$ and $v_w$ are not in the same component of $T\setminus\{w\}$ follows similarly.
\end{proof}

We now characterize all parabolic points in $\boundary (G,\P)$. 

\begin{lem}
\label{gvinfint}
Let $T,V, W$ and $G$ be as in Theorem~\ref{decomp}. Also assume $G_u$ for $u\in \vertices(T)$ is as in Lemma~\ref{T construction and properties}(\ref{T construction and properties: Gw finite}).
Consider a parabolic point $p\in \boundary(G,\P)$ with $\Stab(p)=P$. If $|G_v\cap P|=\infty$ for $v\in V$ then $p\in B_v$. Moreover if $p\in B_v$, then either $p\in w$ for some $w\in v$ or $P\subset G_v$.
\end{lem}
\begin{proof}
If $p\notin B_v$, then since $P$ acts properly discontinuously on $\boundary(G,\P)\setminus \{p\}$ and since $B_v$ is a compact subset in this subspace, Proposition~\ref{Bvprop}(\ref{StabBv}) gives us that $P\cap G_v$ is finite.
Now if $p\in B_v$, suppose there is $g\in P$ such that $g\notin G_v$. Since $gp=p$ and since $gB_v=B_{gv}$ by Proposition~\ref{Bvprop}(\ref{BvGeq}), we have that $p\in B_{v}\cap B_{gv}$. The result then follows from Proposition~\ref{Bvprop}(\ref{Bvpath}).
\end{proof}
\begin{lem}
\label{idealnotpara}
Let $T,V, W$ and $G$ be as in Theorem~\ref{decomp}.
Given a parabolic point $p\in \boundary(G,\P)$, there is $v\in V$ such that $p\in B_v$. Such a $v$ is unique if and only if there is no $w\in W$ such that $p\in w$.
\end{lem}
\begin{proof}
Assume for contradiction that for all $v\in V$ we have that $p\notin B_v$. We first observe that $P$ acting on $T$ has no global fixed point. Lemma~\ref{gvinfint} gives us $P\not\subset G_v$ as $p\notin B_v$ and $P$ is infinite. From Lemma~\ref{T construction and properties}(\ref{T construction and properties: Gw finite}) we get $P\not\subset G_w$ for $w\in W$. Also note that $P$ is finitely generated from Lemma~\ref{T construction and properties}(\ref{T construction and properties: G and P fin gen}).

From the above two observations along with \cite[\S\S I.6.4--I.6.5]{SerreTrees} we can find a $g\in P$ which stabilizes and translates along a line $\ell$ in $T$. Consider some $w\in W\cap \ell$. Let $v$ be the vertex adjacent to $w$ which is closest to $gw$ and let $v'$ be the vertex adjacent to $w$ which is closest to $g^{-1}w$. Then both $v$ and $v'$ lie in $\ell$.

We claim that $p\in \text{cl}(U^w_v)$. Otherwise since $\text{cl}(U^w_v)$ is compact in $\boundary(G,\P)\setminus \{p\}$ and since $g^nw\setminus w\subset \text{cl}(U^w_v)$ for all $n$ by Lemma~\ref{Uvw}(\ref{UvW: direction}), this contradicts that $P$ acts properly discontinuously on $\boundary(G,\P)$. Similarly, $p\in \text{cl}(U^w_{v'})$. Since $\text{cl}(U^w_v)\cap \text{cl}(U^w_{v'})=w$, we have that $p\in w$ which is a contradiction as $w\subset B_v$. 
The second statement follows from Proposition~\ref{Bvprop}(\ref{pair in B_v}) and (\ref{Bvpath}).
\end{proof}

\begin{lem}
\label{Lambda's and P's}
Let $T,V, W$ and $G$ be as in Theorem~\ref{decomp}. Also assume $G_u$ for $u\in \vertices(T)$ is as in Lemma~\ref{T construction and properties}(\ref{T construction and properties: Gw finite}).
For all $w\in W$, let $\Lambda_w=w$ and for all $v\in V$, let 
$\Lambda_v=\{p\in \boundary(G,\P)\mid p\text{ parabolic, }p\in B_v\}$. Moreover for each $u\in \vertices(T)$, let $\P_u=(\Stab(p))_{p\in \Lambda_u}$. Then the following hold.
\begin{enumerate}
    \item \label{Lambda's and P's: P rel hyp}
    For each $u\in U$, the pair $(G_u,\P_u)$ is relatively hyperbolic.
    \item \label{Lambda's and P's: 2-admissible}
    $G$ acts on $T$ with a $2$--admissible action.
    \item \label{Lambda's and P's: s_e map}
    For each $e=(v,w)\in \edges(T)$, there is an inclusion $s_e\colon\Lambda_w\hookrightarrow \Lambda_v$. Moreover the indexed collection $\mathcal{S}=\{s_e\}_{e\in \edges(T)}$ is a signature of the admissible action of $G$ on $T$.
\end{enumerate}
\end{lem}
\begin{proof}

Since the edge stabilizers in $T$ are finite, they are relatively quasiconvex and finitely generated.
Haulmark--Hruska \cite[Proposition~5.2]{HaulmarkHruska_JSJ} prove that if a relatively hyperbolic group acts on a tree such that the edge stabilizers are relatively quasiconvex and finitely generated then so are the vertex stabilizers. Thus $G_v$ is relatively quasiconvex for all $v\in V$.
Moreover Hruska \cite[Theorem~9.1]{Hruska_Relative_Quasiconvex} shows that any relatively quasiconvex subgroup $H$ is hyperbolic relative to the family $\mathcal{Q}(H)$ consisting of all infinite subgroups of the form $H\cap P$ for any $P\in \P$. Furthermore Osin \cite{Osin_Elemenrary_Subgroups} shows that if $(H,\mathcal{Q})$ is relatively hyperbolic and $\mathcal{Q}'$ consists of $\mathcal{Q}$ and finitely many conjugacy classes of finite groups, then $(H,\mathcal{Q}')$ is also relatively hyperbolic.

Now for $v\in V$ let $\P_v=(\Stab(p)\cap G_v)_{p\in \Lambda_v}$. By Lemma~\ref{gvinfint}, it follows that $\mathcal{Q}(G_v)$ is equal to the collection of all infinite members of $\P_v$. Since $\boundary(G,\P)$ is connected, all members of $\P$ are infinite. Then it follows from Lemma~7.5 that if $\Stab(p)\cap G_v$ is finite for some $p\in \Lambda_v$ then $p\in w$ for some $w\in v$. Since $G$ acts by finite quotient on $T$ we have that the finite members of $\P_v$ are in finitely many $G_v$--conjugacy classes. Therefore, $(G_v,\P_v)$ is relatively hyperbolic for all $v\in V$.
Since $G_w$ is finite for all $w\in W$, this proves (\ref{Lambda's and P's: P rel hyp}). 

Various parts of the definition of $G$ acting via a $2$--admissible action on $T$ are satisfied using Lemma~\ref{T construction and properties} and Proposition~\ref{Bvprop}(\ref{BvGeq}), which proves $(\ref{Lambda's and P's: 2-admissible})$. For each $e=(v,w)\in \edges(T)$, Proposition~\ref{Bvprop}(\ref{pair in B_v}) gives us the required inclusion $s_e$. We then get (\ref{Lambda's and P's: s_e map}) from Proposition~\ref{Bvprop}(\ref{BvGeq}).
\end{proof}

Now we can define the parabolic forest $F$ (Definition~\ref{forest}).
We now characterize all parabolic trees in $F$, which were connected components of $F$.

\begin{lem}
\label{paratree}
Let $\Lambda_u$ for $u\in \vertices(T)$ and $\mathcal{S}$ be as in Lemma~\ref{Lambda's and P's}.
Also let $F$ be the parabolic forest for $\mathcal{S}$.
Consider two vertices $u,u'\in \vertices(T)$ along with some $p\in \Lambda_{u}$ and $p'\in \Lambda_{u'}$. Then the vertices $(u,p)$ and $(u',p')$ lie in the same parabolic tree in $F$ if and only if $p=p'$.
\end{lem}
\begin{proof}

If $(u,p)$ and $(u',p')$ lie in the same parabolic subtree they must be joined by a sequence of edges in $F$. Note that by the definition of the signature $\mathcal{S}$ in Lemma~\ref{Lambda's and P's}(\ref{Lambda's and P's: s_e map}), there is an edge in between two vertices $(v,p_1)$ and $(w,p_2)$ in $F$ if and only if there is an edge $e=(v,w)\in T$ and $p_1=p_2$. Thus we must have that $p=p'$.

Now suppose $p=p'$.
If $u\in W$ then there is some $v$ adjacent to $u$ such that $(v,p)$ and $(u,p)$ are joined by an edge.
Therefore we can assume that both $u,u'\in V$ which implies $p\in B_{u}\cap B_{u'}$.
Consider the shortest path $(u=v_1,w_1,v_2,w_2,\ldots, w_{n-1},v_n=u')$ between $u$ and $u'$ in $T$. Then by Proposition~\ref{Bvprop}(\ref{Bvpath}), we have that $p\in w_i$ for $i=1,2,\ldots, n-1$ and $p\in B_{v_i}$ for $i=1,2,\ldots, n$. Therefore, we have edges in $F$ in between the vertices $(v_i,p)$ and $(w_i,p)$ and in between the vertices $(w_i,p)$ and $(v_{i+1},p)$ for $i=1,2,\ldots,n-1$. This gives us a sequence of edges in $F$ connecting $(u,p)$ to $(u',p)$ which proves that they lie in the same parabolic tree in $F$.
\end{proof}

\begin{lem}
\label{paraforest stabs}
Let $F$ be as described in Lemma~\ref{paratree} and $\P$ be the indexed family as described in Theorem~\ref{decomp}. Then $\P$ consists of the stabilizers of parabolic forests in $F$.
\end{lem}
\begin{proof}
By Lemma~\ref{idealnotpara}, for all parabolic points $p\in \boundary(G,\P)$, there is $u\in \vertices(T)$ such that the vertex $(u,p)$ is contained in a parabolic subtree $F_p\subset F$. By Lemma~\ref{paratree}, all vertices in $F_p$ have $p$ as their second coordinate. Moreover every parabolic subtree in $F$ is of this form.
For some $g\in G$, if $p\in w=\Lambda_w$ we have $gp\in gw=\Lambda_{gw}$. By Lemma~\ref{Lambda's and P's}, if $p\in \Lambda_v$ then $p\in B_v$ and for $g\in G$ by Proposition~\ref{Bvprop}(\ref{BvGeq}) we have $gp\in B_{gv}$. Therefore, if $p\in \Lambda_u$ for $u\in \vertices(T)$ then $gp\in \Lambda_{gv}$ for $g\in G$. Thus, by the definition of $F$ we get that $gF_p=F_{gp}$. This implies that $\Stab(F_p)=\Stab(p)$.
\end{proof}

We are now ready to prove Theorem~\ref{decomp}.

\begin{proof}[Proof of Theorem~\ref{decomp}]
Lemma~\ref{T construction and properties}(\ref{T construction and properties: G action on T}) and Lemma~\ref{Lambda's and P's}(\ref{Lambda's and P's: 2-admissible}) give us (\ref{Tminimal}), and  Lemma~\ref{Lambda's and P's}(\ref{Lambda's and P's: P rel hyp}) gives us (\ref{GvRel}).
Lemma~\ref{Lambda's and P's}(\ref{Lambda's and P's: s_e map}) gives us the required signature $\mathcal{S}$. Let $F$ be the parabolic forest for $\mathcal{S}$. Applying Theorem~\ref{combination} with $\mathcal{S}$ we get that there is an indexed family $\mathcal{Q}$ such that $(G,\mathcal{Q})$ is relatively hyperbolic. It follows from Theorem~\ref{combination}(\ref{combination: paraforest}) that $\mathcal{Q}$ is the collection of stabilizers of parabolic forests in $F$. Therefore Lemma~\ref{paraforest stabs} then implies $\mathcal{Q}=\P$, which gives (\ref{Decomp sign}). Finally, Theorem~\ref{boundarytreesys} gives us (\ref{Decomp treesys}).
\end{proof}

\begin{cor}
\label{decompboun}
For each $v\in V$, the boundary $\boundary(G_v,\P_v)$ is the space $B_v$, where $V$ is as in Theorem~\ref{decomp} and $B_v$ is as in Proposition~\ref{Bvprop}.
\end{cor}
\begin{proof}
Recall the tree system $\Theta_G$ (as in Definition~\ref{ThetaG}) with vertex spaces $M_v=\boundary(G_v,\P_v)$ for $v\in V$ and $M_w=w$ for $w\in W$. Then if $\bar{M}_T$ is the completion of the total space $M_T$ as in Setup~\ref{compsetup}, by Theorem~\ref{boundarytreesys} we have a homeomorphism $\bar\Psi\colon \bar{M}_T\rightarrow \boundary(G,\P)$ such that $\bar\Psi\big|_{M_u}$ is the identity for $u\in \vertices(T)$.
Then by Lemma~\ref{Uvw} and Lemma~\ref{insepstr}(\ref{insepstr: tree sys comp}), for $w\in W$ we have the following commutative diagram.
\begin{center}
\begin{tikzcd}
	{\text{Comp}(T\setminus w)} && {\text{Comp}(\partial(G,\mathcal{P})\setminus w)} \\
	{\text{Comp}(\overline{M}_T\setminus M_w)}
	\arrow["{\mu_w}", from=1-1, to=1-3]
	\arrow["{\omega_w}"', from=1-1, to=2-1]
	\arrow["{\overline{\Psi}}"', from=2-1, to=1-3]
\end{tikzcd}
\end{center}
For $v\in V$ and $w\in v$, let $S^w_v$ be as in Proposition~\ref{Bvprop}. Then we get $\bigcap_{w\in v}\omega_w(S^w_v)=M_v$. Thus $\bar\Psi\big|_{M_v}=B_v$ which completes the proof.
\end{proof}



\bibliographystyle{alpha} 
\bibliography{refs} 

\end{document}